%% file: main.tex
\documentclass[11pt]{elsarticle}
\usepackage[letterpaper,top=2cm,bottom=2cm,left=3cm,right=3cm,marginparwidth=1.75cm]{geometry}

\usepackage{graphicx}
\usepackage{float}
\usepackage{tikz}
\usepackage{booktabs}
\usepackage{amsmath, amsthm}
\usepackage{amssymb}    
\usepackage{bm}         
\usepackage{amsfonts}   
\usepackage{mathtools}  
\usepackage{booktabs}   
\usepackage[colorlinks=true, allcolors=blue]{hyperref}
\usepackage[makeroom]{cancel}
\usepackage{algorithm}
\usepackage{algpseudocode}
\usepackage[most]{tcolorbox}
\usepackage{soul} 
\usepackage{vem_nss}
\date{November 2025}

\input{macros}

\begin{document}

\begin{frontmatter}



\title{Analysis of the Virtual Element Approximation for the Smagorinsky model} 

\author[poli,gncs]{Karol L. Cascavita}
\author[poli,gncs]{Francesca Marcon}
\author[poli,gncs]{Maria Strazzullo}

\affiliation[poli]{organization={Politecnico di Torino, Department of Mathematical Sciences ``Giuseppe Luigi Lagrange''},
            addressline={Corso Duca degli Abruzzi, 24},
            city={Turin},
            postcode={10129},
            country={Italy}}

\affiliation[gncs]{organization={INdAM-GNCS group member}, country={Italy}
           }

\begin{abstract}
In this paper, we consider the Smagorinsky model for the Navier-Stokes equations within a virtual element framework. Under the standard assumption of small data, we prove  the existence and uniqueness of a solution. Assuming more regularity to the solutions, we derive the known \del{suboptimal} convergence rates $h$, using $k>1$, for the \textit{a priori} error estimates  of the Smagorinsky model in two dimensional domains. 
We additionally  prove that divergence-free virtual discretizations provide improved convergence orders, with weaker regularity assumptions than in the finite element literature. We conclude the paper with numerical results that corroborate the theory.
\end{abstract}


\begin{highlights}
\item Analysis of the Smagorinsky model via the Virtual Element Method;
\item We derive \textit{a priori} error estimates in two dimensions with the expected suboptimal rates. 
\item We establish the optimal $H^1$-error of order $h^2$ for $k=2$ under weak regularity assumptions.  
\end{highlights}


\end{frontmatter}

\input{intro}

\input{problem}

\input{vem_setting}
\input{discretization}
\input{analysis}
\input{results}

\input{conclusions}

\input{acknowledgements}

\bibliographystyle{abbrvurl}
\bibliography{biblio}

\end{document}

%% file: macros.tex
\newcommand{\bforce}    {{\bm f}}

\newcommand{\bg}    {{\bm g}}

\newcommand{\br}    {{\bm r}}
\newcommand{\bu}    {{\bm u}}
\newcommand{\bv}    {{\bm v}}
\newcommand{\bw}    {{\bm w}}
\newcommand{\bx}    {{\bm x}}
\newcommand{\by}    {{\bm y}}
\newcommand{\bz}    {{\bm z}}

\newcommand{\bI}    {{\bm I}}

\newcommand{\calA}{\mathcal{A}}

\newcommand{\calC}{\mathcal{C}}

\newcommand{\calI}{\mathcal{I}}

\newcommand{\calO}{\mathcal{O}}
\newcommand{\calP}{\mathcal{P}}

\newcommand{\calS}{\mathcal{S}}
\newcommand{\calT}{\mathcal{T}}

\newcommand{\bbQ}{\mathbb{Q}}
\newcommand{\bbU}{\mathbb{U}}
\newcommand{\bbZ}{\mathbb{Z}}

\newcommand{\eI}    {\bm{e}_\calI}
\newcommand{\eH}    {\virtual{\bm{e}}_h}
\newcommand{\uI}    {\bm{u}_\calI}
\newcommand{\pI}    {{p}_\calI}

\newcommand{\Real}      {\mathbb{R}}
\newcommand{\Pol}       {{\mathbb{P}}}

\newcommand{\Div}[1]    {{\text{div}({#1})}}

\newcommand{\Hsymbol}           {\mathrm{H}}
\newcommand{\leblsymbol}[1][2]  {{\mathrm{L}^{{#1}}}}
\newcommand{\Ltwo}              {\leblsymbol}
\newcommand{\Linf}              {\leblsymbol[\infty]}

\newcommand{\Hone}              {{\Hsymbol^1}}
\newcommand{\Ltwozero}          {\mathrm{L}^2_0}

\newcommand{\vHone}[2][\Omega]  {{[\Hsymbol^1({#1})]}^{#2}}
\newcommand{\vHonezero}[2][\Omega]  {{[\Hsymbol^1_0({#1})]}^{#2}}

\newcommand{\norm}[2][]             {\|{#2}\|_{{#1}}}
\newcommand{\seminorm}[2][]         {\left\lvert{#2}\right\rvert_{#1}}

\newcommand{\frobnorm}[1]           {\left|{#1}\right|_{\ell^2}}

\newcommand{\Ltwonorm}[2][\Omega]   {\norm[\Ltwo({#1})]{#2}}
\newcommand{\Linfnorm}[2][\Omega]   {\norm[\Linf({#1})]{#2}}
\newcommand{\Honenorm}[2][\Omega]   {\norm[\Hone({#1})]{#2}}
\newcommand{\Honeseminorm}[2][\Omega]   {\seminorm[\Hone({#1})]{#2}}

\newcommand{\sobh}[2][1]            {\mathrm{H}^{#1}(#2)}

\newcommand{\sobhnorm}[3][1]        { \|{#3}\|_{\sobh[{#1}]{#2}}}
\newcommand{\sobhseminorm}[3][1]    { \left\lvert{#3}\right\rvert_{\sobh[{#1}]{#2}}}

\newcommand \face 	{{F}}
\newcommand \cell 	{{T}}
\newcommand \dCell 	{{\partial \cell}}

\newcommand \mesh	{\mathcal{\cell}_h}

\newcommand \FdCell		{\mathcal{\face}_\dCell}
\newcommand{\hCell}      {h_{\cell}}
\newcommand{\hFace}      {h_{\face}}

\newcommand{\sumMesh}  {\sum_{\cell \in \mesh}}

\newcommand{\PolF}		[1][k]{\Pol^{#1}(\FdCell)}
\newcommand{\PolT}		[1][k]{\Pol^{#1}(\cell)}
\newcommand{\vPoldT}		[1][k]{[\Pol^{#1}(\FdCell)]^2}
\newcommand{\vPolT}		[1][k]{[\Pol^{#1}(\cell)]^{2}}
\newcommand{\tPolT}		[1][k]{[\Pol^{#1}(\cell)]^{2\times 2}}

\newcommand{\GRAD}              {\nabla}

\newcommand{\virtual}[1]    {{#1}}

\newcommand{\ph}            {p_h}
\newcommand{\qh}            {q_h}

\newcommand{\dhv}           {{\virtual{\bm{d}}_h}}
\newcommand{\ehv}           {{\virtual{\bm{e}}_h}}

\newcommand{\uhv}           {{\virtual{\bm{u}}_h}}
\newcommand{\vhv}           {{\virtual{\bm{v}}_h}}
\newcommand{\whv}           {{\virtual{\bm{w}}_h}}
\newcommand{\xhv}           {{\virtual{\bm{x}}_h}}
\newcommand{\yhv}           {{\virtual{\bm{y}}_h}}
\newcommand{\zhv}           {{\virtual{\bm{z}}_h}}
\newcommand{\uhvOne}           {{\virtual{\bm{u}}^1_h}}
\newcommand{\uhvTwo}           {{\virtual{\bm{u}}^2_h}}

\newcommand{\vhvT}          {{\virtual{\bm{v}}_{h|\cell}}}

\newcommand{\ProjSymbol}            {\Pi}
\newcommand{\PiNoper}[1][k]         {{\ProjSymbol}_{#1}^{\nabla}}
\newcommand{\PiZoper}[1][]          {{\ProjSymbol}^{0}_{#1}}
\newcommand{\PiZDoper}[1][k-1]      {{\bm \ProjSymbol}^{0}_{#1}\nabla}
\newcommand{\DPiNoper}[1][k]        {\GRAD \PiNoper[#1]} 
\newcommand{\PiZ}[2][]              {\PiZoper[#1] \virtual{\bm{#2}}_h}
\newcommand{\PiN}[2][k]             {\PiNoper[#1]\virtual{\bm{#2}}_h}
\newcommand{\PiZD}[2][k-1]          {\PiZDoper[#1]\virtual{\bm{#2}}_h}
\newcommand{\DPiN}[2][k]            {\GRAD \PiNoper[#1] \virtual{\bm{#2}}_h} 

\newcommand{\Ph}                        {\calP_h} 
\newcommand{\Phformulation}             {\Pol^{k-1}(\mesh)\bigcap \Ltwozero(\Omega)}

\newcommand{\Vspace}                    {\virtual{\mathcal{V}}}
\newcommand{\Zspace}                    {\virtual{\bbZ}}
\newcommand{\Vh}[1][k]                  {\Vspace^{#1}_h}
\newcommand{\Zh}[1][k]                  {\Zspace^{#1}_h}
\newcommand{\Vcell}[1][k]               {\Vspace^{#1}_{\cell}}

\newcommand{\Mspace}[1][k]              {\mathcal{M}_{#1}(\cell)}
\newcommand{\vMspace}[1][k]             {[\Mspace[#1]]^2}

\newcommand{\smonom}           {{m}}
\newcommand{\vmonom}            {{\bm{\smonom}}}

\newcommand{\dof}[3][]         {\mathlcal{dof}^{{#1}}_{#2}({#3})}

\newcommand{\checkexactarg}[1]
{%
  \ifcat\noexpand#1
    \bm {#1}
  \else
    {#1}
  \fi
}

\makeatletter
\def\checkexact#1{%
  \begingroup
  \def\@arg{#1}%
  \ifcat a\noexpand#1
    \bm{#1}%
  \else
    {#1}%
  \fi
  \endgroup
}
\makeatother

\makeatletter
\def\checkvirtual#1{%
  \begingroup
  \def\@arg{#1}%
  \ifcat a\noexpand#1
     \virtual{\bm{#1}}_h%
  \else
    {#1}%
  \fi
  \endgroup
}
\makeatother
\newcommand{\Aoper}[1][]        {{a}_{#1}}  
\newcommand{\Coper}[1][]        {{c}_{#1}}  
\newcommand{\ShOper}[1][h]      {s_{#1}}

\newcommand{\Ahoper}[1][h]      {\Aoper[#1]}  
\newcommand{\ChOper}[1][h]      {\Coper[#1]}  
\newcommand{\CskewOper}[1][]    {\Coper[#1]^{skew}}  
\newcommand{\ChskewOper}[1][h]  {\Coper[#1]^{skew}}  
\newcommand{\lhoper}[1]         {l_h(\virtual{\bm {#1}}_h)}

\newcommand{\A}[3][]            {\Aoper[#1](\checkexact{#2},\checkexact{#3})}
\newcommand{\Ah}[3][h]          {\Ahoper[#1](\checkvirtual{#2},\checkvirtual{#3})}
\newcommand{\Sh}[3][h]          {\ShOper[#1](\checkvirtual{#2},\checkvirtual{#3})}
\newcommand{\C}[4][]            {\Coper[#1](\checkexact{#2};\checkexact{#3},\checkexact{#4})}
\newcommand{\Ch}[4][h]          {\ChOper[#1](\checkvirtual{#2};\checkvirtual{#3},\checkvirtual{#4})}
\newcommand{\Cskew}[4][]        {\CskewOper[#1](\checkexact{#2};\checkexact{#3},\checkexact{#4})}
\newcommand{\Chskew}[4][h]      {\ChskewOper[#1](\virtual{\bm #2}_h;\virtual{\bm #3}_h,\virtual{\bm #4}_h)}

\newcommand{\TsmagoOper}        {{t}}  
\newcommand{\Thoper}[1][h]      {\TsmagoOper_{#1}}

\newcommand{\Th}[4][h]            {{\Thoper[#1](\checkvirtual{#2}; \checkvirtual{#3}, \checkvirtual{#4})}}

\newcommand{\nuhSmagoOper}      {\nu_S}

\newcommand{\nuhSmago}[1]      {\nuhSmagoOper(\virtual{\bm #1}_h)}
\newcommand{\Cs}               {c_s}
\newcommand{\mapA}[2][] {\calA_{#1}({#2})}

\definecolor{teal}{RGB}{0,128,128}
\definecolor{magenta}{RGB}{255, 0, 255}
\definecolor{ligthblue}{RGB}{80, 161, 225}
\definecolor{red}{rgb}{0.8, 0, 0}
\definecolor{green}{rgb}{0.0, 0.4, 0.0}
\definecolor{vert}{RGB}{95,190,0}
\definecolor{vert2}{rgb}{0.08,0.4,0.}
\definecolor{cement}{RGB}{140,141,135}
\definecolor{lgreyblue}{RGB}{220,230,240}
\definecolor{dgreyblue}{RGB}{110,115,200}
\definecolor{background}{gray}{0.9}
\definecolor{dpipe}{RGB}{0,128,128}
\definecolor{dnode}{RGB}{255, 0, 255}
\definecolor{lbound}{rgb}{0.882353, 0.825071, 0.686107}
\definecolor{dbound}{rgb}{0.582353, 0.525071, 0.386107}
\definecolor{pollito}{RGB}{255, 170, 0.}
\definecolor{vinotinto}{rgb}{0.403922, 0.047059, 0.274510}
\definecolor{tealgreen}{rgb}{0.000000, 0.317647, 0.317647}
\definecolor{lpipe}{rgb}{0.847059, 0.898039, 0.898039}
\definecolor{lnode}{rgb}{0.882353, 0.803922, 0.882353}

\definecolor{teal}{RGB}{0,128,128}
\definecolor{magenta}{RGB}{255, 0, 255}
\definecolor{ligthblue}{RGB}{80, 161, 225}
\definecolor{red}{rgb}{0.8, 0, 0}
\definecolor{green}{rgb}{0.0, 0.4, 0.0}
\definecolor{vert}{RGB}{95,190,0}
\definecolor{vert2}{rgb}{0.08,0.4,0.}
\definecolor{cement}{RGB}{140,141,135}
\definecolor{lgreyblue}{RGB}{220,230,240}
\definecolor{dgreyblue}{RGB}{110,115,200}
\definecolor{background}{gray}{0.9}
\definecolor{cVmesh}{RGB}{225,235,255}
\definecolor{cVmeshdull}{RGB}{120,130,255}
\definecolor{cH1mesh}{RGB}{240,241,235}
\definecolor{cH1meshdull}{RGB}{140,141,135}
\definecolor{dpipe}{RGB}{0,128,128}
\definecolor{dnode}{RGB}{255, 0, 255}
\definecolor{lbound}{rgb}{0.882353, 0.825071, 0.686107}
\definecolor{dbound}{rgb}{0.582353, 0.525071, 0.386107}
\definecolor{pollito}{RGB}{255, 170, 0.}
\definecolor{vinotinto}{rgb}{0.403922, 0.047059, 0.274510}
\definecolor{tealgreen}{rgb}{0.000000, 0.317647, 0.317647}
\definecolor{lpipe}{rgb}{0.847059, 0.898039, 0.898039}
\definecolor{lnode}{rgb}{0.882353, 0.803922, 0.882353}
\definecolor{cyan}{rgb}{0.000000, 0.743137, 0.743137}
\definecolor{teal}{RGB}{0,128,128}
\definecolor{magenta}{RGB}{255, 0, 255}
\definecolor{dmagenta}{RGB}{130, 0, 108}
\definecolor{lteal}{RGB}{150, 200, 220}
\definecolor{ligthblue}{RGB}{80, 161, 225}
\definecolor{red}{rgb}{0.8, 0, 0}
\definecolor{green}{rgb}{0.0, 0.4, 0.0}
\definecolor{vert}{RGB}{95,190,0}
\definecolor{vert2}{rgb}{0.08,0.6,0.}
\definecolor{cement}{RGB}{140,141,135}
\definecolor{lgreyblue}{RGB}{220,230,240}
\definecolor{dgreyblue}{RGB}{110,115,200}
\definecolor{background}{gray}{0.9}
\definecolor{dbulk}{RGB}{0,128,128}
\definecolor{dskel}{RGB}{255, 0, 255}
\definecolor{lbulk}{rgb}{0.847059, 0.898039, 0.898039}
\definecolor{lskel}{RGB}{255, 200, 255}
\definecolor{lbound}{rgb}{0.882353, 0.825071, 0.686107}
\definecolor{dbound}{rgb}{0.582353, 0.525071, 0.386107}
\definecolor{pollito}{RGB}{255, 170, 0.}
\definecolor{vinotinto}{rgb}{0.403922, 0.047059, 0.274510}
\definecolor{tealgreen}{rgb}{0.000000, 0.317647, 0.317647}

\definecolor{lvirtual}{RGB}{220,230,240}
\definecolor{dvirtual}{RGB}{110,115,200}
\definecolor{uvirtual}{RGB}{240,250,250}

\definecolor{karol}{rgb}{0.000000, 0.5, 0.5}

%% file: intro.tex

\section{Introduction}
\label{sec:intro}
The numerical simulation of highly convected fluid flows is a very challenging task across a wide range of industrial and scientific fields. The main challenges arise due to the multiscale and chaotic nature of the phenomenon. In this context,  the meshsize should ensure the accurate capture of the complex flow features \cite{kolmogorov1941localreprint, kolmogorov1941dissipation}. Nevertheless, this requirement leads to prohibitively high computational costs, which have fostered the development of alternative modeling strategies resolving large-scale flow structures and modeling the unresolved subgrid-scale effects.

One of the classical techniques to solve unresolved scales is based on the nonlinear viscosity models. Among them, the earliest and most widely used approach  is the Smagorinsky model. It can be traced back to Joseph Smagorinsky \cite{Smagorinsky1963} and was originally conceived for weather forecasting. The Smagorinsky model augments the Navier-Stokes Equations (NSE) by a dissipative term with a nonlinear eddy viscosity, proportional to the strain rate magnitude. This term follows from physical reasoning consistent with the energy transfer from large to small scale structures. It hinges on the nonnegative property which prevents kinetic energy from growing (under suitable boundary conditions). 
Although the original  model was initially developed as a subgrid-scale closure for time-dependent  flows, its nonlinear eddy-viscosity contribution can also be considered in the stationary  Navier-Stokes system. In this contribution, we focus on the stationary version, henceforth denoted as the Navier-Stokes-Smagorinsky equations (NSSE), for which  we are interested in the mathematical properties and numerical approximation. 

From a mathematical point of view, it acts as a $p$-Laplacian operator by taking $p=3$.  The associated existence of weak solutions  can benefit from the $p$-Laplacian framework developed by the seminal work of  Ladyzhenskaya \cite{Ladyzhenskaya1968} on generalized Navier-Stokes models. We refer the reader to \cite{Guermond2004} and even \cite{Burman2022} for details on the mathematical aspects.

The NSSE have been extensively studied under the Finite Element Method (FEM). Existence and convergence analysis  can be found in the early work of Du and Gunzburger \cite{DuGunzburger1990} for a stationary nonlinear viscosity model linked to the Smagorinsky model. More recently, Shi et al. \cite{ShiLiLi2019} proposed a low order nonconforming mixed finite element method for the stationary Smagorinsky model. They established optimal error estimates along with superconvergence results under suitable regularity assumptions. Their analysis relies on an Aubin--Nitsche duality argument to derive improved velocity error estimates in the $\Ltwo$-norm. Further developments include penalty-based and two-grid finite element approaches can be found in \cite{Yang2023}. We also refer the reader to the monograph of Chacón Rebollo and Lewandowski \cite{BookChacon2014} where the classical \del{$h^{2-d/2}$} error estimates are shown.

To the best of our knowledge, a corresponding theoretical analysis within the Virtual Element Method (VEM) for the Smagorinsky model is still unavailable. Therefore, the goal of the present work is to study the VEM approximation of the NSSE. The devising of the Smagorinsky model within a virtual framework has been recently proposed in \cite{VEM_Smagorinsky2025}, and  investigated through several numerical experiments, including the well-known driven cavity benchmark at high Reynolds numbers. This work provides the theoretical foundation by establishing the well-posedness of the discrete problem and deriving  \textit{a priori} error estimates. 

VEM has been introduced in 2013 in the seminal works \cite{BeiraoBrezziCangianiEtAl2013,BeiraoBrezziMariniEtAl2014} to solve linear diffusion problems. It is now regarded as an extension of the Finite Element Method (FEM) to general polytopal meshes, providing a robust framework that accommodates arbitrary shapes and orders. We refer the reader to \cite{VEM_Review2023} for a comprehensive review. In particular, the divergence-free VEM  enforces exactly the divergence-free constraint, making it suitable  for incompressible flows.  This approach was originally introduced for the Stokes problem in \cite{Antonietti2014} for the lowest-order case; it was extended to arbitrary order in \cite{BeiraoLovadinaVacca2017} and subsequently applied to the Navier-Stokes equations in \cite{BeiraoLovadinaVacca2018}. 
The divergence-free VEM offers two main advantages. First, the incompressibility constraint is guaranteed point-wise by  construction, yielding pressure-robust error estimates for the velocity field. Second, the generality of the polygonal meshes  allows for more complex applications with local adaptation requirements. See for example \cite{VEM_Smagorinsky2025}, where this feature was exploited  to preserve the isotropy of the mesh cells in accordance with the model. Unlike  divergence-free finite elements, such as the Scott-Vogelius element, VEM naturally supports general meshes. 

To the best of our knowledge, this work presents the first theoretical analysis for the divergence-free VEM applied to the steady Navier-Stokes-Smagorinsky equations. Therefore, the main novelties of this contribution are:
\begin{itemize}
\item Under the small-data assumption, we prove the well-posedness of the discrete solution.
\item We derive \textit{a priori} error estimates with the expected  rate $h$, matching the corresponding finite element analysis. The classical result is suboptimal with respect to the polynomial degree $k>1$.
\item We establish an improved optimal $\Hone$-error of order $h^2$ for $k=2$. The result is achieved in the energy norm, without resorting on Aubin-Nitsche duality arguments.  In particular, no additional regularity assumptions are needed when $\bu\in [\sobh[s+1]{{\Omega}}]^2$, with $s\in (1,k]$. Thus,   weaker regularity assumptions are needed in comparison with those commonly employed in the analysis for the Smagorinsky model. 
\end{itemize}

The remainder of the paper is structured as follows. Section \ref{sec:problem} is devoted to the presentation of the continuous problem, while Section  \ref{sec:vem_setting} provides the discrete setting for the divergence-free VEM. Section  \ref{sec:vem_global} briefly presents the discrete global forms and the global problem, including proofs or recalling their relevant properties whenever necessary. Then, we develop the theoretical analysis in Section \ref{sec:theory_wellpos} and Section \ref{sec:theory_apriori}. In particular, the first is devoted to the well-posedness of the discrete formulation of the Navier-Stokes-Smagorinsky problem, while in the latter we provide error estimates. Finally,  Section \ref{sec:results} shows some numerical results, and we finish by drawing some conclusions.

\paragraph{Notation}
Let us introduce the notation that will be used throughout the manuscript.  Let  $\mathcal{O} \subset \mathbb{R}^2$ be a domain.  We remark that, when considering a discrete spatial domain $\Omega$, with this notation, $\mathcal{O}$ might represent the entire computational domain, an element of the discretization, or the boundary of the element.
Let $(\cdot, \cdot)_\calO$ denote the $\Ltwo$-inner product for scalar functions with associated $\Ltwo$-norm $\Ltwonorm[\calO]{\cdot}= \sqrt{(\cdot, \cdot)_\mathcal{O}}$. In particular, the tensor $\Ltwo$-norm is defined as $\Ltwonorm[\calO]{\cdot}^2=\int_\calO \frobnorm{\cdot}^2$, where   $\frobnorm{\cdot}$ denotes the Euclidean norm for vectors or the Frobenius norm defined for matrices in $\mathbb{R}^{d \times d}$. In addition, we adopt the standard notation for the $\Hone$-seminorm and -norm as  $\Honeseminorm[\calO]{\cdot} = \Ltwonorm[\calO]{\nabla (\cdot)}$ and $\Honenorm[\calO]{\cdot} \coloneqq (\Ltwonorm[\calO]{\cdot}^2 + \Honeseminorm[\calO]{\cdot}^2)^{\frac12}$, respectively, for scalar functions. For simplicity, we use the same notation for scalar-,  vector-, and tensor-valued functions, since the intended meaning can be inferred from the context.
Furthermore, it is essential to define the ordered set of natural numbers from $1$ to any integer $n>0$ as $(1:n)$. 
Inside the proofs, we decide to use a single character
$C$ for constants, independent of the mesh size, that appear in the inequalities, which means that we suppose to take at each step the
maximum of the constants involved.

%% file: problem.tex
\section{Problem formulation} \label{sec:problem}

\label{sec:NSE}
We address the incompressible Navier-Stokes Equations (NSE). It reads as: find the velocity $\bu$ and the ratio between the fluid pressure and its density $p$ on the domain $\Omega \subset \mathbb R^{2}$, such that 
\begin{equation}
    \left\lbrace 
    \begin{alignedat}{2}
    - \nu \Delta \bu +
    {(\nabla \bu) \bu} + \nabla p  &= \bforce &\qquad& \text{in $\Omega$}, \\
    \nabla{\cdot} \bu &= 0 &\qquad& \text{in $\Omega$},\\
    \bu &= \bm{0}  &\qquad& \text{on $\partial \Omega,$}
    \end{alignedat} \label{eq:NSE:strong:cont}
    \right. 
\end{equation}

with $\mathbf f$ an external forcing term per unit mass \cite{BookQuarteroni2008numerical} and $\nu>0$ the kinematic viscosity. For simplicity, from now on we refer to $p$ and $\mathbf f$ as the pressure variable and the forcing term, respectively. Homogeneous Dirichlet boundary conditions are applied on the domain boundary $\partial \Omega$. 

Moreover, we define the 
Reynolds number $Re = {\overline{U}{L}}/{\nu}$, where $\overline{U}$, and $L$ are the characteristic velocity, and the characteristic length of the problem. The $Re$ describes the flow regime, as it represents a relationship between inertial and viscous forces: large values of $Re$ indicate a convection-dominated flow with turbulent behavior.

We now introduce the functional framework for the velocity and pressure spaces
as 
\begin{equation*}
\bbU \coloneqq \vHonezero{2} \quad \text{ and } \quad \bbQ \coloneqq \Ltwozero(\Omega) = \{q \in \Ltwo(\Omega) \; : \, (q,1)_\Omega = 0\},  
\end{equation*}
which are equipped with the norm $\|\boldsymbol v \|_{\bbU} \coloneqq \Honeseminorm{\bv}$ for all $ \boldsymbol v\in \bbU$ and $\|q \|_{\bbQ} \coloneqq \|q\|_{\Ltwo(\Omega)}$ for all $q\in \bbQ$.
Classically, the weak formulation of \eqref{eq:NSE:strong:cont} 
reads as follows: find $(\bu, p) \in \bbU \times \mathbb{Q}$ such that:
\begin{equation}
        \left\lbrace 
        \begin{alignedat}{2}
        \nu \A{u}{v} + \C{\bu}{\bu}{\bv} + b( \bv, p)  =& (\bforce,\bv)_\Omega, &\qquad&  \forall \bv \in \bbU , \\
        b(\bu,q) =& 0, &\qquad& \forall q \in \bbQ,
        \end{alignedat} \label{eq:NSE:weak:cont}
        \right. 
\end{equation}
with $\bforce \in [\Ltwo(\Omega)]^2$. Here, the bilinear forms $a: \bbU \times \bbU  \rightarrow \Real$ and $b: \bbU   \times \bbQ \rightarrow \Real$ are defined as 
\begin{eqnarray*}
    \A{w}{v} &\coloneqq& (\nabla \bw, \nabla \bv)_\Omega, \quad \forall \bw,\bv\in\bbU, \\
    b(\bw,q) &\coloneqq& -(q, \nabla \cdot \bw)_\Omega, \quad \forall \bw\in\bbU\;, \forall q\in\bbQ\,,
\end{eqnarray*}   
and the trilinear form $\Coper:  \bbU\times \bbU  \times \bbU \rightarrow \Real$ (
{see}  \cite{layton2008introduction,temam2001navier}) is defined as 
\begin{eqnarray*}
    \C{\bw}{\bz}{\bv} &\coloneqq& ((\GRAD \bz)\bw, \bv)_\Omega\,, \quad \forall \bw,\bv,\bz\in\bbU.
\end{eqnarray*} 

\begin{remark}[{The weak NSE well-posedness}]
    Denoting by $C_{conv}$ the continuity constant of the form $\Coper$, i.e.,
    $$C_{conv} = \sup_{\bv, \bw, \bz \in \bbU\backslash \{\bm 0\}} \frac{\C{w}{z}{v}}{\norm[\bbU]{\bv} \norm[\bbU]{\bw}\norm[\bbU]{\bz}}>0,$$ 
    it was shown by Girault and Raviart in \cite[Theorem 2.2]{BookGiraultRaviart1986} that under the small data assumption 
    \[  
        C_PC_{conv}\frac{\norm[{\Ltwo(\Omega)}]{\bforce}}{\nu^2} < 1 , 
    \]
    a solution $(\bu, p)$ of problem \eqref{eq:NSE:weak:cont} exists and is unique.
    Additionally, the velocity $\bu$ verifies the following bound 
    \begin{equation*}
        \norm[\bbU]{\bu} \leq  C_P\frac{\norm[{\Ltwo(\Omega)}]{\bforce}}{\nu}.
    \end{equation*}
\end{remark}

If we consider the additional hypothesis of a divergence-free velocity field, i.e., $\bu \in \mathbb Z$ with
\begin{equation}
    \bbZ = \{ \bz \in \bbU : \nabla \cdot \bz = 0\},
\end{equation}
the associated weak problem reads:  
find $\bu \in \bbZ$ such that:
\begin{equation}
    \nu a(\bu, \bv) + c(\bu;\bu,\bv) = (\bforce,\bv)_\Omega, \qquad  \forall \bv \in \bbZ, 
    \label{eq:NSE:weak:cont:divfree}
\end{equation}
where $\bu$ is also solution of \eqref{eq:NSE:weak:cont}. 
Given $\bw\in\bbZ$, then the bilinear form $\Coper$ is skew-symmetric, i.e.
\begin{equation*}
    c(\bw;\bz,\bv) = -c(\bw;\bv,\bz) \quad \forall \bz,\bv\in\bbU\,.
\end{equation*}
Hence, we can introduce the trilinear skew-symmetric form $\CskewOper:  \bbU\times \bbU  \times \bbU \rightarrow \Real$  defined as
\begin{eqnarray*}
    \Cskew{\bw}{\bz}{\bv} &\coloneqq& \frac12 \left( \C{\bw}{\bz}{\bv}  - \C{\bw}{\bv}{\bz}  \right),
\end{eqnarray*}    
for every $\bw, \bv, \bz \in \bbU$. 
Then, problem \eqref{eq:NSE:weak:cont:divfree} results as: find $\bu \in \bbZ$ such that:
\begin{equation}
    \nu a(\bu, \bv) + \CskewOper(\bu;\bu,\bv) = (\bforce,\bv)_\Omega, \qquad  \forall \bv \in \bbZ. 
    \label{eq:NSE:weak:cont:divfree:skewsymm}
\end{equation}

%% file: vem_setting.tex
\section{Virtual element setting} \label{sec:vem_setting}
Here, we briefly introduce the enhanced formulation for the divergence-free VEM originally proposed in \cite{BeiraoLovadinaVacca2018}. We begin by presenting some preliminaries regarding the mesh assumptions and the associated discrete setting.

\subsection{Mesh setting}

Let 
$\mesh$ be  
a tessellation of $\Omega$ with polygonal mesh elements denoted by $\cell$. We define the boundary and faces of the element as $\dCell$ and $F$, respectively. The collection of faces lying on $\dCell$ is denoted 
$\FdCell$. In this contribution, the term face is used to indicate a generic edge of a polygonal element.   
The element and the face diameters are indicated with $\hCell$ and $\hFace$, respectively.  Moreover, we define the global meshsize as
\begin{equation}
h\coloneqq\max \limits_{\cell \in \mesh}{\hCell} > 0\label{eq:meshsize}.     
\end{equation}

\begin{assumption}(Mesh assumptions) \label{assum:mesh}
We assume that there exists a real number $\rho >0$, such that the following holds for all $\cell \in \mesh$,  for any $\face\in\FdCell$:
\begin{enumerate}
    \item (\textit{star-shaped property}) 
    $\cell$ is star-shaped with respect to a ball of radius $r_B \geq \rho \hCell$;
    \item (\textit{shape regularity property}) $\hFace$ is comparable with $\hCell$ as $ \rho \hCell \leq \hFace \leq \hCell$.
\end{enumerate}
\end{assumption}
These assumptions can be relaxed: we refer the reader to~\cite{Brenner2018} for further details.

\subsection{Polynomial setting}
Let
$\Pol^k(\mathcal{O})$ be the space of polynomials on a domain $\mathcal{O}$ of degree up to $k$, in particular $\Pol^{-1}(\calO)=\{0\}$. For future use, we also extend this notation to the broken polynomial spaces $\Pol^k(\mesh)$  and $\PolF$ acting on $\mesh$ and $\FdCell$, respectively. Finally, vector polynomial spaces are denoted by $[\Pol^k(\mathcal{O})]^2$.

The associated monomial scalar and vector basis to $\PolT$ and $[\PolT]^2$ are denoted here as $\Mspace\coloneqq\left\lbrace  \smonom_{\bm \alpha} : \frobnorm{\bm \alpha} \leq k  \right\rbrace$  and $\vMspace$, respectively. The scalar scaled monomial $m_{\bm \alpha} \in \Mspace$ with $|\bm \alpha|_{\ell^2} \leq k$ is defined as
\begin{equation}
    m_{\bm \alpha} \coloneqq \prod_{i=1}^2 \left(  \frac{x_i - x^T_i}{\hCell} \right)^{\alpha_i},
    \label{eq:monom:scalar}
\end{equation}
where $\bm \alpha \coloneqq (\alpha_1, \alpha_2)$ follows the standard multi-index notation and $x_i^{\cell}$ denotes the $i$th coordinate of the barycenter of $\cell$. Finally, the vector scaled monomials are $\vmonom_{(\bm \alpha_x, \bm \alpha_y)} \coloneqq (m_{\bm \alpha_x}, m_{\bm \alpha_y}) \in \vMspace$. We refer the reader to the recent review on the proper scaling of normalized monomials and its influence on the conditioning  reported in \cite{Cicuttin2025}.

A key idea underlying the vector formulation is to decompose vector-value polynomials over an element $\cell$ into a gradient component and an orthogonal complement. Henceforth, we restrict to the following Helmholtz-Hodge-like polynomial decomposition 
\begin{equation*}
    {[\PolT]^2 = \nabla \PolT[k+1]   \oplus \bx^\perp\PolT[k-1]}, 
\end{equation*}
with  $\bx^\perp \coloneqq (y, -x)$  \cite{DassiVacca2020}. Namely, by the multiplication action of $\bx^\perp$, we build the orthogonal component to the gradients. 
\subsection{Local projectors}
\label{sec:proj}
For all $\cell \in \mesh$, we define the $\Hone$-projection operator $\PiNoper[k]: \vHone[\cell]{2} \rightarrow \vPolT$ for all $\vhv \in \vHone[\cell]{2}$ as
\begin{equation} 
    \left\lbrace
    \begin{alignedat}{3}
        (\DPiN{v}, \nabla \br)_\cell &= (\nabla \vhv, \nabla \br)_\cell &
     \qquad \forall \ \br \in \vPolT \setminus \vPolT[0],  \label{eq: proj_nabla} \\
        (\PiN{v},  \br)_\dCell &= (\vhv,  \br)_\dCell & \qquad \forall \ \br \in \vPolT[0]\,.
    \end{alignedat} \right. 
\end{equation}
The $\Ltwo$-projection operator $\PiZoper[k]: [\Ltwo(\cell)]^2 \rightarrow \vPolT$ is such that, 
for all $\vhv \in [\Ltwo(\cell)]^2$ 
\begin{align}
    ( \PiZ[k]{v}, \br)_\cell = (\vhv, \br)_\cell & \qquad \forall \ \br \in \vPolT.  
    \label{eq: proj_zero} 
\end{align} 
Finally, proceeding similarly, the $\Ltwo$-projection operator for tensors $\PiZDoper:  [\Ltwo(\cell)]^{2} \rightarrow \tPolT$, 
is defined such that, for all $\vhv \in [\Ltwo(\cell)]^{2}$, it holds
\begin{equation}
    (\PiZD{v}, \bm \tau)_\cell = (\nabla \vhv, \bm \tau)_\cell \qquad \forall \bm \tau \in \tPolT. 
\end{equation}
\subsection{Discrete local and global spaces}

Let us consider a polynomial of degree $k\geq 2$. The  enhanced local virtual space introduced in \cite{BeiraoLovadinaVacca2018} is \\
\begin{equation*}
    \begin{alignedat}{2}
        \Vcell \coloneqq \Big \{ \bv \in 
        \vHone[\cell]{2}    
        :           
            (i)& \quad \bv_{|\dCell}     \, \in \vPoldT \cap [C^0(\partial \cell)]^2 , \\  
            (ii)&\quad - \Delta \bv - \nabla q
                \ \in \bx^\perp\PolT[k-1]  \text{ for some } q \in \Ltwozero(T),\\
            (iii)& \quad \nabla \cdot \bv 
                 \, \in \Pol^{k-1}(\cell), \\
            (iv)& \quad  (\bv - \PiNoper \bv, \bg^\perp)_\cell = 0, \forall \bg^\perp \in  \bx^\perp \PolT[k-1] \setminus \bx^\perp\PolT[k-3]
        \Big \}.  
    \end{alignedat}
\end{equation*}
The local space satisfies the property of polynomial inclusion $\vPolT \subseteq \Vcell$. 
Moreover, the following four subsets constitute a set of Degrees of Freedom (DoFs) for $\Vcell$:
\begin{itemize}
    \item the values of $\bv$:  at the vertices of $\cell$\, and
 at $k-1$ distinct points of every $\face\in\FdCell$\,;
    \item the moments of $\bv$: $(\bv, \bg^\perp)_\cell\,,\, \text{for all }\bg^\perp \in  \bx^\perp \PolT[k-3]$\,;
    \item the moments of $\nabla \cdot \bv $: $(\nabla \cdot \bv, r)_\cell\,,\, \text{for all } r \in \PolT[k-1]\setminus\PolT[0]$\,.
\end{itemize}
In the following, we use the symbol $\dof[\cell]{i}{\cdot}$ ($\dof[h]{i}{\cdot})$ to denote the function that returns the $i$th DoF from a local (global) virtual function.
Moreover, notice that for each $\bv\in\Vcell$ the projectors defined in Section \ref{sec:proj} are computable exploiting the DoFs.

The global virtual spaces for the velocity and the global space for the pressure are defined as follows
\begin{equation*}
        \Vh \coloneqq  \bigtimes_{\cell \in \mesh} \Vcell \bigcap \vHonezero{2} \qquad \wedge \qquad\Ph  \coloneqq \Phformulation,
\end{equation*}
equipped with the norm $\norm[\bbU]{\cdot}$ and $\norm[\bbQ]{\cdot}$, respectively. 
\begin{remark}
 The restriction of the divergence to be a polynomial function of order $k-1$, given by the local condition $(iii)$, results in the important fact that $\Div{\Vh} \subseteq \Ph$, forcing exact mass conservation. Moreover, along with the classical inf-sup condition for the compatibility of pressure and velocity spaces, we actually have that $\Div{\Vh} = \Ph$.
 
In addition, notice that the conformity prescription is only applied to the discrete velocity, while the discrete global pressure is a piecewise polynomial.
\end{remark}

%% file: discretization.tex
\section{Virtual global forms and weak problem} \label{sec:vem_global}
This section is devoted to the presentation of the forms associated with the viscous and convective terms, as well as for the pressure-velocity coupling  form \cite{BeiraoLovadinaVacca2017, BeiraoLovadinaVacca2018}. 
We analyze the turbulent term introduced in \cite{VEM_Smagorinsky2025}. 
This section concludes by showing the discrete global problem for the Navier-Stokes-Smagorinsky equations.
\subsection{Viscous term}
The global bilinear form $a_h: \Vh \times \Vh \rightarrow \Real$ consists on the collection of local contributions as 
    \begin{eqnarray*}
        \displaystyle \Ah{w}{v}  &\coloneqq& \sumMesh \underbrace{(\DPiN{w}, \DPiN{v})_\cell + \Sh[h|\cell]{w}{v}}_{\Ah[h|\cell]{w}{v}},
\end{eqnarray*}    
where $\ShOper[h|\cell]$ denotes the stabilization term, here defined as the standard dofi-dofi formulation 
        \[   \Sh[{h|\cell}]{w}{v} \coloneqq  \sum_{\ell}^{\dim(\Vcell)}  \dof[\cell]{\ell}{(\bI - \PiNoper[k]) \whv}  \dof[\cell]{\ell}{(\bI - \PiNoper[k]) \vhv}, \]
with $\bI$ denoting the identity operator. Summing over elements we obtain the global stabilization term $\ShOper$ needed to ensure coercivity in the discrete system. 

For completeness, we report the following Lemma gathering key properties of $\Ahoper$ from \cite{BeiraoLovadinaVacca2017, BeiraoLovadinaVacca2018}, that will be handy for the theoretical sections.

\begin{lemma}[Properties of $\Ahoper$] The bilinear form $a_h$ has the following properties:
\begin{itemize}
    \item (Coercivity) for all $\whv \in \Vh$, there exists a constant $\alpha>0$, independent of h, such that
        \begin{equation}
            \Ah{w}{w} \geq \alpha \norm[\bbU]{ \whv}^2\label{proper:ah:coercivity};   
        \end{equation}
    \item (Continuity) for all $\whv, \vhv \in \Vh$, there exists a  constant $\zeta_{visc}>0$,  independent of h, such that        \begin{equation}
            \Ah{w}{v} \leq \zeta_{visc} \norm[\bbU]{\whv}\norm[\bbU]{\vhv};\label{proper:ah:continuity} 
        \end{equation}
    \item (k-Consistency)  for all $\cell \in \mesh$, for all $\vhvT \in \Vcell$ and for all $\br  \in \vPolT$ it holds
        \begin{equation}
        \label{proper:ah:poly_consistency}
            \Ahoper[h|T](\vhv,\br) = \A[T]{\vhv}{r}\,,  
        \end{equation}
    with 
    $\Aoper[T]$ denoting the restriction of the global form
    $\Aoper$ to a generic element $\cell$.
\end{itemize}
\end{lemma}

\subsection{Velocity-pressure coupling term}
The global coupling between the velocity and the pressure is made by means of the continuous bilinear form $b$, such that for all $\vhv \in \Vh$, the global contribution consists on the collection of local contributions as 
\begin{equation}
    \displaystyle b(\vhv, \qh) = \sumMesh b_T(\vhv, \qh) = -\sumMesh (\qh, \nabla \cdot {\vhv})_\cell,    \qquad \forall \qh \in \Ph. \label{proper:b}
\end{equation}
As before, we report the main properties of the form following  \cite{BeiraoLovadinaVacca2017}. 
\begin{lemma}[Properties and conditions on $b$] The bilinear form $b$ satisfies the following inequalities:
    \begin{itemize}
        \item  (Continuity) for all $(\vhv, \qh) \in \Vh \times \Ph$, it holds
        \begin{equation}
            \big| b(\vhv,\qh) \big| \leq \norm[\bbU]{\vhv} \Ltwonorm{\qh}; \label{proper:b:continuity}
        \end{equation} 
        \item (The inf-sup condition): there exists a constant $\beta>0$ independent of the meshsize $h$, such that
    \begin{equation}
         \sup_{\vhv \in \Vh \setminus \{\bm 0\}} \frac{b(\vhv, \qh)}{\norm[\bbU]{\vhv}} \geq \beta \| \qh \|_{\Ltwo(\Omega)},   \quad \forall \qh \in \Ph. \label{proper:b:inf-sup}
    \end{equation}
    \end{itemize}
\end{lemma}
    
\subsection{Convective term}
The discrete trilinear form $\ChOper:  \Vh\times \Vh \times \Vh \rightarrow \Real$ related with the nonlinear convective term is defined such that, for all $(\xhv, \yhv, \zhv) \in \Vh\times \Vh \times \Vh$,
\begin{equation}
    \Ch{x}{y}{z} \coloneqq \sumMesh (({\PiZD{y}}) \PiZ[k]{x}, \PiZ[k]{z})_\cell.    
\end{equation}
Alternatively, analogously to the continuous level, the discrete nonlinear term can be approximated by the skew-symmetric version   $\ChskewOper:  \Vh\times \Vh \times \Vh \rightarrow \Real$ defined such that, for all $(\xhv, \yhv, \zhv) \in \Vh\times \Vh \times \Vh$,
\begin{equation} 
    \label{eq:}
    \Chskew{x}{y}{z} = \frac12\left(  \Ch{x}{y}{z} -  \Ch{x}{z}{y}\right).
\end{equation}
Linear forms can be recast in alternative formulations to the ones presented here, by employing different projectors. In the following, we use the skew-symmetric setting, for which we present the next Lemma 
\cite{ BeiraoLovadinaVacca2018}.
\begin{lemma}[Properties of $\ChskewOper$]
The trilinear form $\ChskewOper$ has the following properties:
\begin{itemize}
    \item  (Skew-symmetry) for all $\xhv, \yhv, \zhv \in \Vh$, the following holds true
            \begin{equation}
                \Chskew{x}{y}{z} = - \Chskew{x}{z}{y}; \label{proper:chskew:skew-symmetry}
            \end{equation}
    \item (Non-dissipativity)
    for all $\xhv, \yhv \in \Vh$ the following holds true
            \begin{equation}
                \Chskew{x}{y}{y} = 0;
                \label{proper:chskew:zero}
            \end{equation}
    \item (Continuity) for all $\bx, \by, \bz \in \bbU$, there exists 
        \begin{equation}
            {\zeta}_{conv}\coloneqq \sup_{\bx,\by,\bz \in \bbU \backslash \{\bm 0\}} \frac{\ChskewOper(\bx;\by,\bz)}{\norm[\bbU]{\bx} \norm[\bbU]{\by} \norm[\bbU]{\bz}}>0,\label{proper:chskew:continuity:constant}
        \end{equation}
        independent of the meshsize $h$, such that the bound holds
        \begin{equation}
            \ChskewOper(\bx;\by,\bz) \leq {\zeta}_{conv} \norm[\bbU]{\bx} \norm[\bbU]{\by} \norm[\bbU]{\bz}.  \label{proper:chskew:continuity}
        \end{equation}      
\end{itemize}
\end{lemma}
The proof can be directly derived from \cite[Proposition 3.3]{BeiraoLovadinaVacca2018}.

\subsection{Smagorinsky term}
Compared to the formulation of the standard NSE, the Smagorinsky model adds an Eddy viscosity term to the discrete formulation. The turbulent term  is strictly related to the discretization of the space and relies on the Smagorinsky constant \cite{SamueleThesis}, which is classically set (guided by numerical experiments \cite{lilly1966representation}) as $\Cs = 0.1$: in this paper we follow this standard choice. 
The discretized Smagorinsky form $\Thoper: \Vh \times \Vh \times \Vh \rightarrow \Real$, using a VEM approach, is written as follows \cite{VEM_Smagorinsky2025}
\begin{eqnarray}
        \displaystyle \Th{x}{y}{z} &\coloneqq& \sumMesh (\nuhSmago{x}\DPiN{y}, \DPiN{z})_\cell,  \label{eq:smago:discrete:term}  
\end{eqnarray} 
where the Smagorinsky viscosity $\nuhSmagoOper$ is also recast into a discrete VEM framework and defined  such that 
\begin{equation}
\label{eq:smago:discrete:viscosity}
    \nuhSmago{x} \coloneqq \Cs^2\sumMesh \hCell^2 \frobnorm{\PiZD{x}}
    \chi_{\cell},
\end{equation}
where $\chi_\cell$ is the characteristic function over the tessellation element. We remark that $\nu_S$ accounts for small amount of turbulent viscous effects with $\calO(h^2)$. Thus, for $\hCell \rightarrow 0$, the Eddy viscosity term $\nuhSmago{\bu} \rightarrow 0$. This totally complies with the Kolmogorov scale turbulence description: a refined mesh can capture the complex behavior of the flow and does not need stabilization, which is instead needed for a larger meshsize.

To the best of our knowledge, this is the first tome that an analysis on the Smagorinsky term is carried in the VEM framework. We propose uniqueness and existence of the discrete problem together with \emph{a priori} error estimates. We recover the classical FEM estimate in Theorem \ref{theo:estimates:velocity}, while we prove that, under proper additional regularity assumptions, the order of convergence increases, reaching even expected NSE results for $k=2$, see Theorem \ref{theo:estimates:velocity:higher_regularity}. We now present the properties of the Smagorinsky-related terms in the following lemma and prove it right after.   
\begin{lemma}[Properties of $\nuhSmagoOper$ and $\Thoper$] \label{proper:as}
    The discrete VEM Smagorinsky term has the following properties:
    \begin{itemize}
        \item (Nonnegativity of $\nuhSmagoOper$) given $\Cs>0$, for all $\xhv \in \Vh$, the following holds true
            \begin{equation}
                \nuhSmago{x} \coloneqq \Cs^2\sumMesh \hCell^2 \frobnorm{\PiZD{x}}  \geq 0;     \label{proper:nus:nonnegativity}
            \end{equation}        
        \item (Continuity of $\nuhSmagoOper$) there exists a constant $\zeta_{\nu}>0$ independent of $h$, such that for all  $\xhv \in \Vh$ the following estimate holds  
        \begin{equation}
            \Linfnorm{\nuhSmago{x}}  \leq \zeta_{\nu} h \norm[\bbU]{\xhv} \label{proper:vs:bound};
        \end{equation}    
    \item (Nonnegativity of $\Thoper$)  let $\xhv \in \Vh$, then for all $\yhv \in \Vh$, it holds
        \begin{equation}
            \Th{x}{y}{y} \geq 0\label{proper:as:nonnegativity};     
        \end{equation}  
            \item (Continuity of $\Thoper$)  
            there exists a constant $\zeta_{smag}>0$, independent of $h$
             such that for all $\xhv, \yhv, \zhv \in \Vh$, it holds
        \begin{equation}        
            |\Th{x}{y}{z}| \leq \zeta_{smag} h \norm[\bbU]{\xhv}\norm[\bbU]{\yhv}\norm[\bbU]{\zhv};
            \label{proper:as:bound}
        \end{equation}
    \item (Lipschitz-continuity of $\Thoper$): 
    there exists a constant $L_{smag}>0$ such that for all $\xhv, \yhv, \zhv \in \Vh$ the following holds true
        \begin{eqnarray}
               \Big| \Th{x}{y}{z} - \Th{w}{y}{z} \Big| \leq L_{smag} \norm[\bbU]{\xhv - \whv} \norm[\bbU]{\yhv} \norm[\bbU]{\zhv}\,,\label{proper:as:lipschitz}
        \end{eqnarray}     
        where $L_{smag}=C_{inv}\Cs^2 h$.
    \end{itemize}
\end{lemma}

\begin{proof}[Proof of the nonnegativity of $\nuhSmagoOper$]
The nonnegativity trivially derives from the definition.
\end{proof}
\begin{proof}[Proof of the continuity of $\nuhSmagoOper$]
 For all $\xhv \in \Vh$, we infer from the definition of $\nuhSmagoOper$ in \eqref{eq:smago:discrete:viscosity} that 
\begin{align*}
      \Linfnorm[\Omega]{\nuhSmago{x}}  &=  \Cs^2 \Linfnorm[\Omega]{\sumMesh \hCell^2  | \PiZD{x}|_{\ell^2}} 
    \leq 2\Cs^2 h^{2}\sumMesh   \Linfnorm[\cell]{ \PiZD{x}}, \\
    &\leq2 C_{inv} \Cs^2 h^2\sumMesh  \hCell^{-1} \Ltwonorm[\cell]{\PiZD{x}}
    \leq 2C_{inv} \Cs^2 h \sumMesh  \Ltwonorm[\cell]{\GRAD \xhv},\\
    &\leq \zeta_{\nu} h \norm[\bbU]{\xhv}, 
\end{align*}
where we have used the discrete inverse inequality (see \cite[Lemma 1.50]{BookDiPietroErn2012}) 
and the continuity of the projection with respect to the $\Ltwo$-norm in the second line. 
\end{proof}

\begin{proof}[Proof of the nonnegativity of $\Thoper$]
Invoking the nonnegativity of $\nuhSmagoOper$ in \eqref{proper:nus:nonnegativity}, and applying the definition of  $\Thoper$, the claim readily follows. 
\end{proof}
 
\begin{proof}[Proof of the continuity of $\Thoper$] 
We apply the generalized H\"{o}lder inequality, we continue using property \eqref{proper:vs:bound} and the continuity of the projection with respect to the $\Hone$-seminorm, yielding that for all $\xhv, \yhv, \zhv \in \Vh$ it holds 
    \begin{align*}
        |\Th{x}{y}{z}| 
        &\leq \sumMesh \Linfnorm[\cell]{\nuhSmago{x}} \Ltwonorm[\cell]{\DPiN{y}}\Ltwonorm[\cell]{\DPiN{z}} \leq   \zeta_\nu h\norm[\bbU]{\xhv}\norm[\bbU]{\yhv}\norm[\bbU]{\zhv}.   
    \end{align*}
\end{proof}

\begin{proof}[Proof of Lipschitz-continuity of $\Thoper$]
For all $\xhv, \yhv, \zhv \in \Vh$, using the definitions of the Smagorinsky term \eqref{eq:smago:discrete:term} and the Smagorinsky viscosity \eqref{eq:smago:discrete:viscosity}, we  infer that
\begin{align*}   
   \Big| \Th{x}{y}{z} - &\Th{w}{y}{z} \Big| 
   \\
   &\leq \Cs^2 h^2 \sumMesh  \left| \Big((\frobnorm{\PiZD{x}} - \frobnorm{\PiZD{w}}) \DPiN{y}, \DPiN{z} \Big)_\cell \right|\\&=:\Cs^2 h^2 \sumMesh \tau. 
\end{align*}
Now, applying the triangle inequality for integrals, we estimate the local contributions as follows  
\begin{align*} 
   \tau
   &\leq \Big(\Big| \frobnorm{\PiZD{x}} - \frobnorm{\PiZD{w}}\Big| \Big|\DPiN{y}\Big| ,  \Big|\DPiN{z}\Big| \Big)_\cell
   \\&\leq \Big(\frobnorm{\PiZDoper(\xhv-\whv)}  \Big|\DPiN{y}\Big| , \Big|\DPiN{z}\Big| \Big)_\cell
   \\
   &\leq C_{inv}h_T^{-1}\Ltwonorm[\cell]{\PiZDoper(\xhv - \whv)} \Ltwonorm[\cell]{\DPiN{y}}\Ltwonorm[\cell]{\DPiN{z}}\\
    &\leq C_{inv}h_T^{-1}\Honeseminorm[\cell]{\xhv - \whv} \Honeseminorm[\cell]{\yhv}\Honeseminorm[\cell]{\zhv},
\end{align*}
where we have invoked the reverse triangle inequality in the second line, the Cauchy-Schwarz inequality in the third line together with a discrete inverse inequality  and the continuity of the projections in the fourth line. Hence, we conclude by summing over the mesh elements and we obtain the estimate \eqref{proper:as:lipschitz} with $L_{smag} = C_{inv}\Cs^2 h$.  
\end{proof}

\begin{remark}[Three-dimensional case]
Even though in this document we do not treat specifically the 3D case, our main results hold, up to constants, as the proofs rely on properties of the forms (coercivity, continuity), and inf-sup conditions. In particular, \textit{a priori} error estimates for the three-dimensional scales with $h^{1/2}$, according to the continuity of the Smagorinsky term. We refer also the reader to the appendix of \cite{BeiraoLovadinaVacca2017} for the 3D VEM version. 
\end{remark}

\subsection{Discrete global weak formulation} 
 
We consider now the discrete version of the weak problem \eqref{eq:NSE:weak:cont} including the Smagorinsky diffusive term that reads as: find $(\uhv, \ph) \in \Vh \times \Ph $
such that
\begin{equation}
    \left\lbrace 
    \begin{alignedat}{2}
    & \nu \Ah{u}{v} + \Th{u}{u}{v}  +  \Chskew{u}{u}{v}+ b(\vhv, \ph)  =   l_h(\vhv), &\qquad&  \forall \vhv \in \Vh , \\
    &  b(\uhv,\qh) = \ 0, &\qquad& \forall \qh \in \Ph,
    \end{alignedat} \label{eq:smago:weak:discrete}
    \right. 
\end{equation}
with global linear form $l_h: \Vh \rightarrow \Real$ formulated as $\lhoper{v} \coloneqq (\bforce_h,\vhv)_{\Omega}$, and $\bforce_h$ is defined such that its restriction to an element $\cell \in \mesh$ is as 
\begin{equation}
    \bforce_{h|\cell} \coloneqq \PiZoper[k]\bforce.    \label{eq:force:discrete:def}
\end{equation}
Let us now introduce the space of weakly divergence-free velocities
\begin{equation}
    \Zh \coloneqq \{\vhv \in \Vh: \forall \qh \in \Ph, b(\vhv,\qh) = 0  \}. 
\end{equation}
The system \eqref{eq:smago:weak:discrete} can be reformulated as:  
find $\uhv \in \Zh$ such that
\begin{equation}
    \begin{alignedat}{2}
    & \nu \Ah{u}{v} + \Th{u}{u}{v} 
    +  \Chskew{u}{u}{v} =   l_h(\vhv), &\qquad&  \forall \vhv \in \Zh , 
    \end{alignedat} \label{eq:smago:weak:discrete:divfree}
\end{equation}
where $\uhv$ is also a solution of system  \eqref{eq:smago:weak:discrete}.

%% file: analysis.tex
\section{Theoretical analysis: well-posedness of the discrete formulation}
\label{sec:theory_wellpos}
In this section, we address the analysis of the weak problem \eqref{eq:smago:weak:discrete}, including the action of the Smagorinsky term. More precisely, we study its well-posedness.

\begin{lemma}[Existence of velocity] \label{lemma:existence:velocity}   The weak problem of the Navier-Stokes-Smagorinsky equations of \eqref{eq:smago:weak:discrete:divfree} admits at least one solution $\zhv \in \Zh$, satisfying the following a priori bound   
\begin{equation}
    \norm[\bbU]{\zhv} \leq (\alpha \nu)^{-1}C_P {\|\bforce}\|_{\Ltwo(\Omega)},\label{eq:existence:velocity:bound}
\end{equation}
with $\alpha$ the coercivity constant of $\Ahoper$ from \eqref{proper:ah:coercivity} and $C_P$ is the Poincaré constant.
\end{lemma}
\begin{proof}
    Given $\whv\in\Zh$, let $\mapA[\whv]{\cdot}:\Zh\to\mathbb{R}$ be the linear operator defined as for all $\vhv\in\Zh$
    \begin{align}
         \mapA[\whv]{\vhv}
         \coloneqq \nu \Ah{w}{v} + \Th{w}{w}{v}  +  \Chskew{w}{w}{v} - l_h(\vhv). \label{eq:map_phi}
     \end{align}
     Notice that $\mapA[\whv]{\cdot}$ is a continuous operator, since it is the sum of continuous operators (see \eqref{proper:ah:continuity}, \eqref{proper:as:bound}, \eqref{proper:chskew:continuity}).
     Then, by the Riesz–Fréchet Theorem (see for instance \cite[Theorem 5.2]{Rynne2008}) we get that there exists a unique $\zhv\in\Zh$ such that $\mapA[\whv]{\vhv}=(\zhv,\vhv)_\Omega$ for all $\vhv\in\Zh$. 
     Hence, we can define the continuous operator $\phi:\Zh\to\Zh$ that maps each $\whv$ to the corresponding $\zhv$.  Then, for every $\whv\in\Zh$ we estimate from below the quantity $(\phi(\whv),\whv)_\Omega$, i.e., applying  the properties \eqref{proper:chskew:zero} and \eqref{proper:as:nonnegativity}
     \begin{align}
       (\phi(\whv),\whv)_\Omega =   \mapA[\whv]{\whv} &= \nu\Ah{w}{w} + \Th{w}{w}{w} + \Chskew{w}{w}{w}- l_h(\whv) \nonumber \\ 
    & \geq \nu \Ah{w}{w}  - l_h(\whv)  \nonumber \\
    & \geq \alpha \nu \norm[\bbU]{\whv}^2 - C_P \Ltwonorm{\bforce}\norm[\bbU]{\whv}\,,
     \end{align}
     where in the last step we use the coercivity of the bilinear form $\Ahoper$  \eqref{proper:ah:coercivity}, the Cauchy-Schwarz inequality, and the Poincaré inequality.
     These results imply that, given the closed sphere $\calS_r=\{ \xhv \in \Vh:  \norm[\bbU]{\xhv} \leq r \}$ with radius $r = (\alpha \nu)^{-1}C_P\Ltwonorm{\bforce}$,
\begin{equation}
      (\phi(\whv),\whv)_\Omega  \geq 0 , \qquad \forall \whv \in \partial \calS_r.  \label{eq:existence:nonnegativity} 
\end{equation}
We conclude by invoking the Fixed-point Theorem, that states the existence of a fixed point $\uhv \in \Zh$ such that $(\phi(\uhv),\vhv)_\Omega  = 0$ for all $\vhv \in \Zh$. This yields the conclusion that the thesis is correct.
\end{proof}
\begin{lemma}[Existence of pressure]   
 The weak problem of the Navier-Stokes-Smagorinsky equations of \eqref{eq:smago:weak:discrete} admits at least one solution $(\uhv,\ph) \in \Vh\times\Ph $,
 satisfying the following a \textit{priori} bound   
\begin{equation}
    \Ltwonorm{\ph}\leq \beta^{-1} \left(\frac{\zeta_{visc}}{\alpha} + \left(1 + \frac{\zeta_{smag}h + \zeta_{conv}}{(\alpha\nu)^2}\right)C_P\norm[\Ltwo(\Omega)]{\bforce} \right)C_P\norm[\Ltwo(\Omega)]{\bforce} \label{eq:existence:pressure:bound}.
    \end{equation}
\end{lemma}
\begin{proof}
The existence of the solution $(\uhv,\ph) \in \Vh\times\Ph $ of the equations \eqref{eq:smago:weak:discrete} is a direct consequence of \eqref{proper:b:inf-sup} and Lemma \ref{lemma:existence:velocity}. Then, let us estimate $|b(\vhv, \ph)|$, using the Problem \eqref{eq:smago:weak:discrete}, the boundedness of the forms \eqref{proper:ah:continuity}, \eqref{proper:as:bound} \eqref{proper:chskew:continuity} and Lemma \ref{lemma:existence:velocity}, we get
\begin{equation}
\begin{aligned}
 | b(\vhv, \ph)| &=|  \nu \Ah{u}{v} + \Th{u}{u}{v}  +  \Chskew{u}{u}{v}  -  l_h(\vhv)|
 \\
&\leq \left( \nu\zeta_{visc}\norm[\bbU]{\uhv} + (\zeta_{smag}h + \zeta_{conv})\norm[\bbU]{\uhv}^2 + C^2_P\norm[\Ltwo(\Omega)]{\bforce}^2\right)\norm[\bbU]{\vhv}\\
        &\leq \left(  \frac{\zeta_{visc}C_P}{\alpha}\Ltwonorm{\bforce} + \left(1 + \frac{\zeta_{smag}h + \zeta_{conv}}{(\alpha\nu)^2}\right)C^2_P\Ltwonorm{\bforce}^2 \right)\norm[\bbU]{\vhv}.
        \end{aligned}
\end{equation}
    Considering the inf-sup stability \eqref{proper:b:inf-sup} and the previous estimate, we obtain
    the thesis \eqref{eq:existence:pressure:bound}.
\end{proof}

\begin{lemma}[Uniqueness] Assume 
the following bound is verified
\begin{equation}
        \kappa \coloneqq C_P\frac{\Ltwonorm{\bforce}}{\alpha^2 \nu^2}  \Big( \zeta_{conv} + L_{smag}\Big) < 1, \label{eq:uniqueness:velocity}
\end{equation}
with $\alpha, \zeta_{conv}$ and $L_{smag}$ as in \eqref{proper:ah:coercivity}, \eqref{proper:chskew:continuity:constant} and \eqref{proper:as:lipschitz}. Then, problem \eqref{eq:smago:weak:discrete} has a unique solution  $(\uhv,\ph) \in \Vh\times \Ph$.
\end{lemma}

\begin{proof}
    Let $\uhvOne, \uhvTwo \in \Zh$  be two solutions of \eqref{eq:smago:weak:discrete:divfree}. 
    Thus, taking as test function $\dhv = \uhvOne- \uhvTwo$ in the weak problem \eqref{eq:smago:weak:discrete:divfree},  we readily have
    \begin{eqnarray}
        \nu\Ahoper(\virtual{\bm u}_h^1,\virtual{\bm d}_h)&=& -\Thoper(\virtual{\bm u}_h^1;\virtual{\bm u}_h^1,\virtual{\bm d}_h) 
         - \ChskewOper(\virtual{\bm u}_h^1;\virtual{\bm u}_h^1,\virtual{\bm d}_h)+ l_h(\dhv), \\
         \nu\Ahoper(\virtual{\bm u}_h^2,\virtual{\bm d}_h)&=& -\Thoper(\virtual{\bm u}_h^2;\virtual{\bm u}_h^2,\virtual{\bm d}_h)  - \ChskewOper(\virtual{\bm u}_h^2;\virtual{\bm u}_h^2,\virtual{\bm d}_h)
              + l_h(\dhv)\,.
    \end{eqnarray}
    Subtracting the previous equations, considering  the coercivity  of $\Ahoper \;$  \eqref{proper:ah:coercivity} and the nonnegativity property of $\Thoper$ \eqref{proper:as:nonnegativity}, we infer that
    \begin{equation} \label{eq:proof-unic-step1}
        \begin{aligned}
         \alpha \nu \norm[\bbU]{\dhv}^2   &\leq  \nu\Ah{d}{d} +   \Thoper(\virtual{\bm u}_h^1;\virtual{\bm d}_h,\virtual{\bm d}_h)
         \\
         &= \nu\Ahoper(\virtual{\bm u}_h^1,\virtual{\bm d}_h) - \nu\Ahoper(\virtual{\bm u}_h^2,\virtual{\bm d}_h) +   \Thoper(\virtual{\bm u}_h^1;\virtual{\bm u}_h^1,\virtual{\bm d}_h)
         -   \Thoper(\virtual{\bm u}_h^1;\virtual{\bm u}_h^2,\virtual{\bm d}_h)
         \\
         &= \Thoper(\virtual{\bm u}_h^2;\virtual{\bm u}_h^2,\virtual{\bm d}_h)  -   \Thoper(\virtual{\bm u}_h^1;\virtual{\bm u}_h^2,\virtual{\bm d}_h) + \ChskewOper(\virtual{\bm u}_h^2;\virtual{\bm u}_h^2,\virtual{\bm d}_h)
         -\ChskewOper(\virtual{\bm u}_h^1;\virtual{\bm u}_h^1,\virtual{\bm d}_h)\,.
        \end{aligned}
    \end{equation}
    Now let us consider the relation
    \begin{equation}\label{eq:relCskewProof}
        \ChskewOper(\virtual{\bm u}_h^2;\virtual{\bm u}_h^2,\virtual{\bm d}_h) = \ChskewOper(\virtual{\bm u}_h^2;\virtual{\bm u}_h^1,\virtual{\bm d}_h),
    \end{equation}
    obtained by
    \begin{equation}
        \begin{aligned}
          \ChskewOper(\virtual{\bm u}_h^2;\virtual{\bm u}_h^2,\virtual{\bm d}_h) &=  - \ChskewOper(\virtual{\bm u}_h^2;\virtual{\bm d}_h,\virtual{\bm u}_h^2) 
          \\
          &= -\ChskewOper(\virtual{\bm u}_h^2;\virtual{\bm u}_h^1,\virtual{\bm u}_h^2) + \ChskewOper(\virtual{\bm u}_h^2;\virtual{\bm u}_h^2,\virtual{\bm u}_h^2)
          \\
          &= -\ChskewOper(\virtual{\bm u}_h^2;\virtual{\bm u}_h^1,\virtual{\bm u}_h^2) + \ChskewOper(\virtual{\bm u}_h^2;\virtual{\bm u}_h^1,\virtual{\bm u}_h^1)
          \\
          &= \ChskewOper(\virtual{\bm u}_h^2;\virtual{\bm u}_h^1,\virtual{\bm d}_h),
        \end{aligned}
    \end{equation}
    where we have used properties \eqref{proper:chskew:skew-symmetry} and \eqref{proper:chskew:zero} of $\ChskewOper$.
    Hence, using relation \eqref{eq:relCskewProof} in \eqref{eq:proof-unic-step1}, we get
    \begin{equation}\label{eq:proof-unic-step2}
    \begin{aligned}
        \alpha \nu \norm[\bbU]{\dhv}^2   
        &\leq \Thoper(\virtual{\bm u}_h^2;\virtual{\bm u}_h^2,\virtual{\bm d}_h)  -   \Thoper(\virtual{\bm u}_h^1;\virtual{\bm u}_h^2,\virtual{\bm d}_h) -\ChskewOper(\virtual{\bm d}_h;\virtual{\bm u}_h^1,\dhv)
               \\
        &\leq  L_{smag}\norm[\bbU]{\dhv}^2 \norm[\bbU]{\virtual{\bm u}_h^2} +\zeta_{conv} \norm[\bbU]{\dhv}^2\norm[\bbU]{\virtual{\bm u}_h^1} \\
        & \leq \norm[\bbU]{\dhv}^2    \Ltwonorm{\bforce} C_P(\alpha \nu)^{-1} \Big( \zeta_{conv} + L_{smag} \Big)\,,
    \end{aligned}
    \end{equation}
    where we have applied the Lipschitz continuity of $\Thoper$   \eqref{proper:as:lipschitz}, the continuity of $\ChskewOper$ \eqref{proper:chskew:continuity} and the bound \eqref{eq:existence:velocity:bound} both for $\virtual{\bm u}_h^1$ and $\virtual{\bm u}_h^2$.

    Let us define 
\[ 
    \kappa \coloneqq C_P\frac{{\|\bforce}\|_{\Ltwo(\Omega)}}{\alpha^2 \nu^2}  \Big( \zeta_{conv} + L_{smag}\Big).
\]
Assuming  $0 < \kappa < 1$ then $\dhv = 0$ and the uniqueness of the velocity solution of \eqref{eq:smago:weak:discrete:divfree}.
Finally, the uniqueness of the pressure derives immediately from the second equation of \eqref{eq:smago:discrete:term}. 
\end{proof}

\begin{remark}[Uniqueness bound with alternative eddy viscosity]
Let us use an alternative definition of the eddy viscosity  \eqref{eq:smago:discrete:viscosity} in terms of the operator $\GRAD \PiNoper$ with  $   \nuhSmago{x}_{|\cell} \coloneqq \Cs^2 \hCell^2 \frobnorm{\DPiN{x}}$. Let  $\virtual{\bm u}_h^1, \virtual{\bm u}_h^2 \in \Vh$ be two solutions of problem \eqref{eq:smago:weak:discrete}. Let also the bound $\Thoper(\virtual{\bm u}_h^1;\virtual{\bm u}_h^1,\virtual{\bm d}_h) -\Thoper(\virtual{\bm u}_h^2;\virtual{\bm u}_h^2,\virtual{\bm d}_h)\geq 0$ hold, as a benefit of the strong monotonicity of the Smagorinsky term, see \cite[Lemma 8.88]{John2016}.  Thus, following the analysis in the previous lemma and under slight modifications, it is straightforward to obtain the uniqueness bound as \[\kappa \coloneqq C_P{\|\bforce}\|_{\Ltwo(\Omega)}(\alpha \nu)^{-2} \zeta_{conv}<1.\] 
\end{remark}


\section{Theoretical analysis: a priori error estimates}
\label{sec:theory_apriori}
In this section, we derive the \textit{priori} error estimates for the proposed scheme.

\subsection{Preliminaries to existence, uniqueness and convergence results}

\begin{lemma}[Approximation properties] \label{lemma:approx}
See \cite{BrambleHilbert,DupontScott}. 
    Let Assumption \ref{assum:mesh} hold. Let $\PiZoper[k]$ be the scalar $L^2$-orthogonal projection into $\mathbb{P}_k(\cell)$ and $\PiNoper$ the scalar $H^1$-orthogonal projection. Then, there exists $C > 0$, 
    independent of $\hCell$, such that for all $\phi \in H^{s+1}(\cell)$ with $s\geq 0$ the following bounds hold
    \begin{equation}
    \sobhseminorm[m]{\cell}{ \phi - \PiZoper[k]\phi}
    \leq
    C \hCell^{s+1-m} \sobhseminorm[s+1]{\cell}{\phi},
    \qquad \text{with } m \leq s \leq k, \; \text{with } m\geq 0,
    \label{lemma:approx:pi_zero}
    \end{equation}
    
    
    \begin{equation}
    \sobhseminorm[m]{\cell}{  \phi - \PiNoper[k]\phi}
    \leq
    C \hCell^{s+1-m}, \sobhseminorm[s+1]{\cell}{\phi}
    \qquad \text{with } m \leq s \leq k, \; \text{with } m\geq 0.
    \label{lemma:approx:pi_nabla}
    \end{equation}
\end{lemma}
For the vector-valued functions in our analysis, we can proceed by reasoning component-wise to extend the previous results. 
\begin{lemma}[$\Vh$-approximation properties]\label{lemma:vh_approx} See \cite[Theorem 4.1]{BeiraoLovadinaVacca2018}. 
Let Assumption \ref{assum:mesh} hold. Then, for all $\cell \in \mesh$ and for all $\bv \in \bbU\bigcap [\sobh[s+1]{\Omega}]^d$ with $s\in[0,k]$, for a constant $C>0$, there exists $\bv_\calI \in \Vh$ holding   
\begin{equation} 
   \Ltwonorm[\Omega]{\bv - \bv_{\calI}} + \hCell\Honeseminorm[\Omega]{\bv - \bv_{\calI}} \leq C \hCell^{s+1} \sobhseminorm[s+1]{\Omega}{\bv}  \label{eq:vh_approx_prop} 
\end{equation} 
\end{lemma}

For the sake of completeness, we first recall some known results on error estimates. 

\begin{lemma}[Forcing error estimate]  \label{lemma:estimate:force} See \cite[Lemma 4.5]{BeiraoLovadinaVacca2018}.  Assume that $\bforce \in [\sobh[s+1]{\Omega}]^d$  with $s \in [-1, k]$ and $\bforce_h$ defined as in \eqref{eq:force:discrete:def}. Then, there exists a constant $C>0$ such that for all $\vhv \in \Vh$ the following holds true 
  \begin{equation}
      \Big|(\bforce - \bforce_h, \vhv)_\Omega \Big|
      \leq C h^{s+2} \sobhnorm[s+1]{\Omega}{\bforce} \Ltwonorm[\Omega]{\vhv}.
  \end{equation} 
\end{lemma}

\begin{lemma}[Consistency error estimate of $\ChskewOper$] \label{lemma:pre:c:consistency} 
  See \cite[Lemma 4.3]{BeiraoLovadinaVacca2018}. Assume that $\bx \in [\sobh[s+1]{\Omega}]^d\cap \bbU$ with $s \in [0,k]$. Then, there exists a constant $C >0$, such that for all $\by \in \bbU$, the following holds true   
  \begin{equation}
      \Big|\Cskew{x}{x}{y} - \ChskewOper(\bx;\bx,\by)\Big|
      \leq  C h^s \calC(\bx) \norm[\bbU]{\by},  
  \end{equation} 
  with $\calC(\bx) = \left(\sobhnorm[s]{\Omega}{\bx} 
                                + \norm[\bbU]{\bx} 
                                + \sobhnorm[s+1]{\Omega}{\bx} 
                             \right)
                             \sobhnorm[s+1]{\Omega}{\bx}$.
\end{lemma}

\begin{lemma}[Alternative  continuity of $\ChskewOper$] \label{lemma:pre:ch:continuity_alternative} 
  See \cite[Lemma 4.4]{BeiraoLovadinaVacca2018}. Let $\zeta_{conv}$ be the constant in \eqref{proper:chskew:continuity:constant}. Then, for all $\bx, \by,\bz \in \bbU$ the following holds true 
  \begin{equation}
      \Big| \ChskewOper(\bx,\bx,\bz) - \ChskewOper(\by,\by,\bz)\Big|
      \leq  \zeta_{conv} \Big(\norm[\bbU]{\by} \norm[\bbU]{\bz} + \norm[\bbU]{\bx-\by + \bz} (\norm[\bbU]{\bx} + \norm[\bbU]{\by}) \Big)\norm[\bbU]{\bz}.  
  \end{equation}  
\end{lemma}

\begin{lemma} \label{lemma:pre:convection}
  Let $\zeta_{conv}$ be the constant in \eqref{proper:chskew:continuity:constant}. Let Lemma \ref{lemma:pre:c:consistency} hold.  Then, there exists a constant $C>0$ such that for all $\bx \in \bbU$ and for all $\xhv, \zhv \in \Vh$ the following holds true
     \begin{align}
         \left| \Chskew{x}{x}{z} - \Cskew{\bx}{\bx}{\zhv}\right| \leq C \Honeseminorm{\zhv}\left(\left(\Honeseminorm{\xhv} + \Honeseminorm{\bx}\right)\right.&  \Honeseminorm{\bx - \xhv}\nonumber \\ &+ h^s\calC(\bx)\Big)  
     \end{align}
     with $\calC(\bx)$ defined as in Lemma \ref{lemma:pre:c:consistency}. Alternatively if $\zhv = \bx_h - \bx_I$, with $\bx_I$ defined as in  
    \begin{align}
         \left| \Chskew{x}{x}{z} - \Cskew{\bx}{\bx}{\zhv}\right| \leq C \Honeseminorm{\zhv}&\left(\Honeseminorm{\xhv}\Honeseminorm{\zhv}  + h^s\calC(\bx)\Honeseminorm{\xhv} \right. \nonumber\\& \left. +\Honeseminorm{\bx-\bx _I}\left(\Honeseminorm{\bx} + \Honeseminorm{\xhv} \right) \right)   
     \end{align}
\end{lemma}
\begin{proof}
    This is a direct consequence of Lemma \ref{lemma:pre:c:consistency} and Lemma \ref{lemma:pre:ch:continuity_alternative}. One can proceed as in \cite[Proof of Proposition 5.1, step 5]{Antonietti2022} (for instance). 
\end{proof}

\subsection{Error analysis}
This section is devoted to our main results that contain the convergence rates of VEM for the velocity and the pressure applied to the NSS problem.

\begin{theorem}[Velocity error  estimate] \label{theo:estimates:velocity}
Let $\bu \in \bbU\cap [\sobh[s+1]{{\Omega}}]^2$, with $s\in [0,k],$ and $\uhv \in \Vh$ be the solutions of the continuous problem \eqref{eq:NSE:weak:cont} and the discrete problem \eqref{eq:smago:weak:discrete:divfree}, respectively. Assume that bound \eqref{eq:existence:velocity:bound} holds. 

Then, there exists $\calC_u(\bu,\bforce,\nu)>0$, independent of h, such that 
\begin{equation}
    \norm[\bbU]{\bu - \uhv}  \leq \calC_u(\bu, \bforce, \nu) h.\label{eq:estimate:velocity}
\end{equation}
\end{theorem}
\begin{proof}
Let $\eI \coloneqq \bu - \uI$ and $\eH\coloneqq \uI - \uhv$ with $\uI$ defined as in Lemma \ref{lemma:vh_approx}.  
Using the triangle inequality and the interpolation estimate \eqref{eq:vh_approx_prop}, we get
\begin{equation}
    \norm[\bbU]{\bu - \uhv} \leq 
    \norm[\bbU]{\eI}
    + \norm[\bbU]{\eH}
    \leq C_I h^{s} \sobhseminorm[s+1]{\Omega}{\bu} + \norm[\bbU]{\eH} \label{eq:estimate:velocity:error_triag}.
\end{equation}
Hence, it remains to estimate $\Honeseminorm{\eH}$.
We start by using the coercivity of $\Ahoper$, adding and subtracting the term $\A{u}{\ehv}$, the systems \eqref{eq:NSE:weak:cont:divfree} and \eqref{eq:smago:weak:discrete:divfree} using as test function $\eH$, yielding
\begin{align}
    \alpha\nu\norm[\bbU]{\ehv}^2\leq \nu\Ah{e}{e}&=\underbrace{\nu\Ah{\uI}{\ehv}  - \nu\A{u}{\ehv}}_{\calT_a \text{ viscous term}} + \underbrace{\Th{u}{u}{e}}_{\calT_S \text{ turb. term}}  \nonumber\\
   & + \underbrace{\Chskew{u}{u}{e}-c^{skew}(\bu, \bu,\ehv)}_{\calT_c \text{ convective term}} + \underbrace{ (\bforce - \bforce_h, \ehv)}_{\calT_f \text{ force term}}. \label{eq:estimate:velocity:proof:pde_cont_discrete} 
\end{align}
Following \cite{BeiraoLovadinaVacca2018}, we infer that 
\begin{itemize}
    \item \textit{Viscous term:} we restrict the analysis of these terms for now to local contributions, denoting with $a_T$ the restriction of $a$ to an element $T$. 
\begin{align*}
     \calT_{a| \cell}
    &= \nu\Ahoper[h|\cell](\uI,\ehv)- \nu\A[\cell]{u}{\ehv} \pm \nu\Ah[h|\cell]{\bu_\pi}{\ehv} \\
    &=\nu\Ah[h|\cell]{\uI - \bu_\pi}{\ehv} + \nu\A[\cell]{\bu_\pi - \bu}{\ehv} \lesssim \nu\hCell^{s} \sobhseminorm[s+1]{\cell}{\bu} \Honeseminorm[\cell]{\ehv}, 
\end{align*}
where we have added and subtracted a suitable term 
in the first line, we have used the polynomial consistency of $\Ahoper$ \eqref{proper:ah:poly_consistency} to pass to the second line, we have applied  the continuity of $\Ahoper$, 
the triangle inequality, and finally,  we have invoked the approximation properties in Lemma \ref{lemma:approx} and Lemma \ref{lemma:vh_approx}. 
\item \textit{Force term:} we invoke Lemma \ref{lemma:estimate:force} to estimate the local contributions of $\calT_f$. 
\item \textit{Convective term:} the nonlinear term 
$\calT_c$ is bounded applying Lemma \ref{lemma:pre:convection} to get 
\begin{align*}
    \calT_c &\leq \zeta_{conv}\norm[\bbU]{\ehv}\left( \norm[\bbU]{\ehv} \norm[\bbU]{\uhv} + \norm[\bbU]{\eI}\left(\norm[\bbU]{\bu} + \norm[\bbU]{\uhv} \right)+ h^s\calC(\bu)\right) \nonumber\\
    & \leq \zeta_{conv}\norm[\bbU]{\ehv}\left(\norm[\bbU]{\ehv} \norm[\bbU]{\uhv} + h^s\sobhseminorm[s+1]{\Omega}{\bu}\left(\norm[\bbU]{\bu} + \norm[\bbU]{\uhv} \right) + h^s\calC(\bu)\right)\\
    & = \zeta_{conv} \norm[\bbU]{\uhv}\norm[\bbU]{\ehv}^2 +\zeta_{conv}h^s \underbrace{\left(\calC(\bu) +  \sobhseminorm[s+1]{\Omega}{\bu}\left(\norm[\bbU]{\bu} +\norm[\bbU]{\uhv} \right) \right)}_{=:\calC^\star(\bu)} \norm[\bbU]{\ehv}.
\end{align*}
\item \textit{Turbulent term:} the Smagorinsky term  is bounded using the continuity \eqref{proper:as:bound}. Whence
\begin{align*}
    \calT_S &\leq  \zeta_{smag} h \norm[\bbU]{\uhv}^2\norm[\bbU]{\ehv}\nonumber. 
\end{align*}
\end{itemize}

Summing over the mesh elements the local terms $\tau_{a,\cell}$, and  the above estimates, we obtain
\begin{align*}
    \nu \alpha \norm[\bbU]{\ehv}^2 & \leq \norm[\bbU]{\ehv}\left(\sumMesh \nu \tau_{a,|\cell}+\calT_S + \calT_c + \calT_f \right)\nonumber  \\
    &\lesssim \nu h^s \sobhseminorm[s+1]{\Omega}{\bu}\norm[\bbU]{\ehv} + \zeta_{smag}  h\norm[\bbU]{\uhv}^2 \norm[\bbU]{\ehv} \nonumber 
     \\
    &\quad + \zeta_{conv} \norm[\bbU]{\uhv}\norm[\bbU]{\ehv}^2 + \zeta_{conv}h^s \calC^\star(\bu)  \norm[\bbU]{\ehv} + h^{s+2} \sobhnorm[s+2]{\Omega}{\bforce}\norm[\bbU]{\ehv}.
\end{align*}
Dropping $\norm[\bbU]{\ehv}$, gathering the remaining terms on the left hand side in the coefficient $\gamma$ as follows
\[
\gamma(\bu, \bforce, \nu) = \nu \alpha - \zeta_{conv} \norm[\bbU]{\uhv} \geq \nu \alpha -  \zeta_{conv}\frac{\Ltwonorm{\bforce}}{\alpha \nu}
\] 
Where, we have used the bound \eqref{eq:existence:velocity:bound} for the discrete velocity. Assuming $\nu\alpha$ to be large enough such that $\gamma>0$ 
\begin{align}
    \gamma\norm[\bbU]{\ehv} & \leq \nu h^s \sobhseminorm[s+1]{\Omega}{\bu} +  \zeta_{smag} \frac{\Ltwonorm{\bforce}^2}{(\alpha\nu)^2} h +  \zeta_{conv}h^s\calC^\star(\bu) +Ch^{s+2} \sobhnorm[s+2]{\Omega}{\bforce}. \label{eq:estimate:velocity:eh}
\end{align}
Now, introducing \eqref{eq:estimate:velocity:eh} into \eqref{eq:estimate:velocity:error_triag} we get 
\begin{align*}
\norm[\bbU]{\bu - \uhv}&\leq h^s \left(\left(C_I  + \frac{\nu }{\gamma}\right)\sobhseminorm[s+1]{\Omega}{\bu}   + \frac{\zeta_{conv}\calC^\star}{\gamma} \right)   +  \frac{\zeta_{smag}}{\gamma} h\frac{\Ltwonorm{\bforce}^2}{(\alpha \nu)^2}
    + \frac{C}{\gamma} h^{s+2} \sobhnorm[s+2]{\Omega}{\bforce}.
\end{align*}
Whence the assertion.
\end{proof}

\begin{theorem}[Smooth solutions in 2D]\label{theo:estimates:velocity:higher_regularity}
Let $\uhv \in \Vh$ be the solution of the discrete problem \eqref{eq:smago:weak:discrete:divfree}, satisfying the bound \eqref{eq:existence:velocity:bound}. 
Moreover, let $\bu$ be the solution of the continuous problem \eqref{eq:NSE:weak:cont}.
Assume $\bu\in \bbU\cap[\sobh[s+1]{{\Omega}}]^2$, with $s\in (1,k]$. Then,  there exists $\calC_u(\bu,\bforce,\nu)>0$, independent of h, such that 
\begin{equation}
    \norm[\bbU]{\bu - \uhv}  \leq \calC_u(\bu, \bforce, \nu) h^2.\label{eq:estimate:velocity:double}
\end{equation}
Whereas, if $\bu \in \bbU\cap [\sobh[s+1]{{\Omega}}]^2$, with $s\in [0,1],$ and $\nabla \bu\in[\Linf(\Omega)]^{2\times 2}$. Then, there exists $\calC_u(\bu,\bforce,\nu)>0$, independent of h, such that 
\begin{equation}
    \norm[\bbU]{\bu - \uhv}  \leq \calC_u(\bu, \bforce, \nu) h^{1+s}.\label{eq:estimate:velocity:sp1}
\end{equation}
\end{theorem}
\begin{proof}
    We can follow the same structure as the proof of Theorem \ref{theo:estimates:velocity}. 
    The modifications concern only the turbulent term who determines the total convergence rate.
    Hence,  using $\uhv = \uI - \ehv$
    and the nonnegativity property \eqref{proper:as:nonnegativity} we get
    \begin{equation*}
        \Th{u}{u}{e} =\Th{\uhv}{\uI}{\ehv} -\Th{u}{e}{e} \leq \Th{\uhv}{\uI}{\ehv}\,.
    \end{equation*}
Then, recalling the definition of $\Thoper$ \eqref{eq:smago:discrete:term}, we apply the generalized H\"older inequality, the continuity of the projectors, and the estimate \eqref{eq:existence:velocity:bound} to  obtain
\begin{equation}\label{eq:proof:estim1}
\begin{aligned}
       \Th{\uhv}{\uI}{\ehv} 
       &\leq h^2 \Cs^2 \sumMesh  \Ltwonorm[\cell]{|\PiZD{u}|_{\ell^2}} \norm[\Linf(\cell)]{\DPiNoper \uI}\Ltwonorm[\cell]{\DPiN{e}}
       \\
       & = h^2 \Cs^2 \sumMesh  \Ltwonorm[\cell]{\PiZD{u}} \norm[\Linf(\cell)]{\DPiNoper \uI}\Ltwonorm[\cell]{\DPiN{e}}
       \\
       &\leq h^2 \Cs^2 \sumMesh  \Ltwonorm[\cell]{\nabla \uhv} \norm[\Linf(\cell)]{\DPiNoper \uI}\Ltwonorm[\cell]{\nabla \eH}
       \\
       &\leq h^2 \Cs^2 \frac{C_P \Ltwonorm{\bforce}}{\alpha \nu}    \left(\sumMesh \norm[\Linf(\cell)]{\DPiNoper \uI}\right) \norm[\bbU]{\eH}\,.
\end{aligned}
\end{equation}
Now, let us estimate the term 
\begin{equation}\label{eq:proof:estim2}
    \begin{aligned}
        \sumMesh \norm[\Linf(\cell)]{\DPiNoper \uI} 
        &\leq   \sumMesh \left( \norm[\Linf(\cell)]{\DPiNoper \bu} + \norm[\Linf(\cell)]{\DPiNoper (\bu-\uI)} \right)
        \\
        &\leq \sumMesh \left( C_{cont} \norm[\Linf(\cell)]{\nabla \bu} + C_{inv,T} h_T^{-1}\Ltwonorm[\cell]{\nabla(\bu-\uI)}\right)
        \\
        &\leq  C_{cont}\norm[\Linf(\Omega)]{\nabla \bu} + C_{inv} h^{s-1} \sobhseminorm[s+1]{\Omega}{\bu}\,,
    \end{aligned}
\end{equation}
where we applied the discrete inverse estimate, the continuity of the projectors and Lemma \ref{lemma:vh_approx}. 
We remark that if $\bu\in [\sobh[s+1]{{\Omega}}]^2$, with $s\in (1,k]$, $\nabla \bu\in[\Linf(\Omega)]^{2\times 2}$ by the Sobolev embedding.

Let us consider together \eqref{eq:proof:estim1} and \eqref{eq:proof:estim2} to get
\begin{equation}
     \Th{\uhv}{\uI}{\ehv}  \leq h^2 \Cs^2 (\alpha \nu)^{-1}C_P {\|\bforce}\|_{\Ltwo(\Omega)} \left( C_{cont}\norm[\Linf(\Omega)]{\nabla \bu} + C_{inv} h^{s-1} \sobhseminorm[s+1]{\Omega}{\bu}\right) \norm[\bbU]{\eH}\,.
\end{equation}
Hence, if $s\geq 1$ the term $h^{s-1}$ is negligible and the order of convergence $h^2$ is recovered. Whereas, if $s\in[0,1)$ the order obtained is $h^{1+s}$. 
\end{proof}

\begin{theorem}[Pressure error estimate]Let $(\bu, p)\in \bbU\times \bbQ$ and $(\uhv, \ph) \in \Vh \times \Ph$ be the solutions of the continuous problem \eqref{eq:NSE:weak:cont} and the discrete problem \eqref{eq:smago:weak:discrete}, respectively. Assume that the bounds \eqref{eq:existence:velocity:bound} and \eqref{eq:existence:pressure:bound} hold. Moreover, assume additional regularity for $(\bu,p) \in [\sobh[s+1]{\Omega}]^2\times \sobh[s]{\Omega}$. 
Then, there exist $\calC^{high}(\bu, p, \nu)>0$ and $\calC^{low}(\bu,\bforce, \nu)>0$, independent of h, such that 
\begin{equation}
    \Ltwonorm[\Omega]{p - \ph} \leq \calC^{high}(\bu, p, \nu)h^s + \calC^{low}(\bu,\bforce, \nu)h,
\end{equation}
with coefficients of high order and low order defined as follows, respectively
\begin{align*}
     C^{high}(\bu, p, \nu) &\coloneqq  \beta^{-1}\left(C\nu \sobhseminorm[s+1]{\Omega}{\bu} + C\calC(\bu) + \sobhseminorm[s]{\Omega}{p}\right),\\
     C^{low}(\bu, \bforce, \nu) &\coloneqq \beta^{-1}\left( \nu \alpha_{cont} +  \zeta_{conv}\left(\norm[\bbU]{\bu} + \frac{\Ltwonorm{\bforce}}{\alpha \nu}\right) \right)\calC_u(\bu,\bforce,\nu) + \beta^{-1}C\zeta_{smag} \frac{\Ltwonorm{\bforce}^{2}}{(\alpha \nu)^{2}}.
\end{align*}
\end{theorem}
\begin{proof}
Let $p_I$ be defined such that $p_{I,\cell}=\PiZoper[k-1]p$ for all $\cell$ in $\mesh$. Let us define the approximation error and the discretization error for the pressure as $e_\calI=p-p_\calI$ and $e_h = p_\calI - p_h$, respectively. 
As for the velocity estimate, we use the triangle inequality and by polynomial approximation properties, we readily obtain 
\begin{equation}
    \Ltwonorm{p - \ph} \leq 
    \Ltwonorm{e_I}
    + \Ltwonorm{e_h}
    \leq C_I h^{s}  \sobhseminorm[s]{\Omega}{p} + \Ltwonorm{e_h} \label{eq:estimate:pressure:error_triag}.
\end{equation}
Hence, it remains to estimate $\Ltwonorm{e_h}$, 
    \begin{align*}
        b(\vhv, e_h) &= b(\vhv, \pI)  - b(\vhv, \ph)  \pm b(\vhv, p) \\   
        &= b(\vhv, p)  - b(\vhv, \ph)  - b(\vhv, e_\calI)\\
        &= \nu\Ah{u}{v} + \Th{u}{u}{v} + \Chskew{u}{u}{v}-l_h(\vhv) \\
        &\qquad -\nu\A{u}{\vhv} - \Cskew{u}{u}{\vhv} + (\bforce, \vhv)_{\Omega} - b(\vhv, e_\calI)\\
        &= \underbrace{\nu(\Ah{u}{v}  -\A{u}{\vhv})}_{\calT_a \text{ viscous term}}  + \underbrace{\Th{u}{u}{v}}_{\calT_S  \text{ turb. term}} +\underbrace{(\Chskew{u}{u}{v}-\Cskew{u}{u}{\vhv})}_{\calT_c \text{ convective term}}  \\
        & \qquad + \underbrace{(\bforce - \bforce_h, \vhv)_{\Omega}}_{\calT_f \text{ forcing term}} - \underbrace{b(\vhv, e_\calI)}_{\calT_b \text{  coupling term}}.
    \end{align*}
By repeating  the calculations in  the proof of Lemma \ref{theo:estimates:velocity} with 
test function $\vhv$ instead of $\ehv$, all these terms can be estimated, except 
for $\calT_b$. 
\begin{align*}
    \calT_a &\leq \nu\alpha_{cont} \norm[\bbU]{\bu -\uhv} \norm[\bbU]{\vhv} + C\nu h^s \sobhseminorm[s+1]{\Omega}{\bu}\norm[\bbU]{\vhv}, \\
    \calT_S &\leq C\zeta_{smag}  h\Ltwonorm{\bforce}^2(\alpha \nu)^{-2}\norm[\bbU]{\vhv},\\
    \calT_c&\leq
    \zeta_{conv} \norm[\bbU]{\vhv}\norm[\bbU]{\bu -\uhv}\left( C\Ltwonorm{\bforce}(\alpha \nu)^{-1} + \norm[\bbU]{\bu}\right)+ C\norm[\bbU]{\vhv}h^s\calC(\bu),  
\end{align*}
where $\calT_c$ was estimated by summing and subtracting the terms $\ChskewOper(\bu;\bu,\vhv)$ and $ \ChskewOper(\bu;\uhv,\vhv)$
\begin{align*}
    \calT_c &\leq   (\Chskew{u}{u}{v}-\Cskew{u}{u}{\vhv}) \pm \ChskewOper(\bu;\bu,\vhv) \pm \ChskewOper(\bu;\uhv,\vhv)\\
    &=:\tau_{c,1}+\tau_{c,2}\pm\tau_{c,3}\pm\tau_{c,4},
\end{align*}
and further gathered and estimated as follows
\begin{align*}
    (\tau_{c,1}-\tau_{c,4}) + (\tau_{c,4}-\tau_{c,3}) &\leq 
    \zeta_{conv}(\norm[\bbU]{\bu} + \norm[\bbU]{\uhv})\norm[\bbU]{\bu -\uhv}\norm[\bbU]{\vhv},\\
    (\tau_{c,3}-\tau_{c,2}) &\leq  C h^s \calC(\bu) \norm[\bbU]{\vhv},
\end{align*}
where we used the continuity of $\ChskewOper$ \eqref{proper:chskew:continuity} in the first line and applied the Lemma \ref{lemma:pre:c:consistency} to the second.
Turning to $\calT_b$, owing to the continuity \eqref{proper:b:continuity} of the bilinear form $b$, we obtain
\begin{equation}
   b(\vhv, e_\calI) \leq \ \norm[\bbU]{\vhv} \Ltwonorm{p -\pI} \leq \norm[\bbU]{\vhv}(h^s \sobhseminorm[s]{\Omega}{p}).
\end{equation}
Gathering terms
\begin{align*}
     b(\vhv, e_h) 
     &\leq h^s\left( C\nu \sobhseminorm[s+1]{\Omega}{\bu} + C\calC(\bu) + \sobhseminorm[s]{\Omega}{p}\right) \norm[\bbU]{\vhv}\\
     & \quad + h\left(\left( \nu \alpha_{cont} +  \zeta_{conv}\left(\norm[\bbU]{\bu} + \frac{\Ltwonorm{\bforce}}{\alpha \nu}\right) \right)\calC_u(\bu,\bforce,\nu) + C\zeta_{smag} \frac{\Ltwonorm{\bforce}^{2}}{(\alpha \nu)^{2}} \right) \norm[\bbU]{\vhv}.
\end{align*}
Where we reorganized terms and used the estimates for the discrete velocity \eqref{eq:existence:velocity:bound} and the velocity error \eqref{eq:estimate:velocity}. 
To conclude, the proof relies on the inf-sup condition \eqref{proper:b:inf-sup} and the previous estimate
    \begin{equation*}
         \| p_\calI-\ph \|_{\Ltwo(\Omega)} \leq \beta^{-1}\sup_{\vhv \in \Vh \setminus \{\bm 0\}} \frac{b(\vhv,e_h)}{\norm[\bbU]{\vhv}}.
    \end{equation*}
\end{proof}

%% file: results.tex
\section{Numerical results}
\label{sec:results}
In this section, we validate the proposed estimates in two  analytical test cases.  We consider the model problem \eqref{eq:NSE:strong:cont} solved through the Newton method. In both cases, the domain is the unitary square $\Omega$. Let  $({\mesh})_h$  be an $h$-refined mesh sequence of the domain $\Omega$. 
As a validation metrics, we consider the $\Hone$-errors for velocity and the $\Ltwo$-error for the pressure, along with the convergence rates for different the Reynolds number.
For clarity, we report errors metrics: 
\begin{align*}
\|\nabla \bu - \PiZD[k-1]{u}\|_{\mesh}^2 & \coloneqq  \sum_{\cell \in \mesh} \| \nabla \bu - \PiZD[k-1]{u} \|^2_{\Ltwo(\cell)}, \\
\| p - p_h \|_{\mesh}^2 & \coloneqq  \sum_{\cell \in \mesh} \| p - p_h \|^2_{\Ltwo(\cell)}.
\end{align*}
 In the tests, we consider degree $k \in \{2,3,4\}$, and the cartesian, hexagonal and triangular meshes depicted in Figure \ref{fig:meshes}, for $h=\sqrt{2}/4, 6\sqrt{3/2}, 1/2$, respectively.

\begin{figure}[H]
    \centering
    \includegraphics[width=0.25\linewidth]{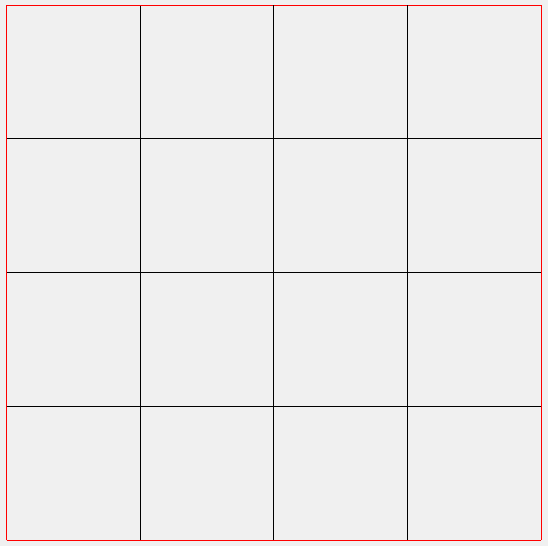}
    \includegraphics[width=0.25\linewidth]{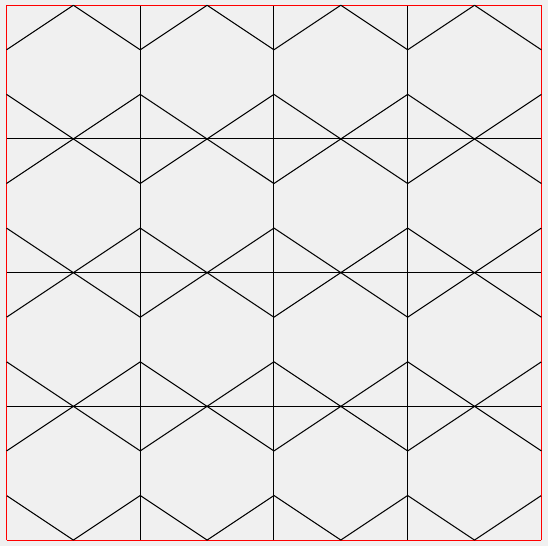}
    \includegraphics[width=0.25\linewidth]{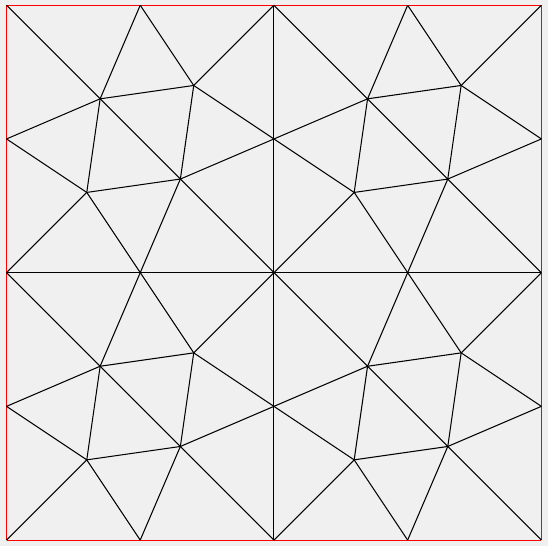}
    \caption{Meshes: cartesian, 
    hexagonal and triangular for $h=\sqrt{2}/4, 6\sqrt{3/2}, 1/2$, respectively.}
    \label{fig:meshes}
\end{figure}
The two test cases are denoted as:
\begin{itemize}
    \item \emph{Sinusoidal test}: let us take the functions $u_x = 0.5 \sin^2(2 \pi x) \sin(2 \pi y) \cos(2 \pi y)$
    and $u_y=-0.5 \sin^2(2 \pi y) \sin(2 \pi x) \cos(2 \pi x)$. Then, the exact velocity field is defined as $\bu(\bx) = (u_x, u_y )$ satisfying the homogeneous Dirichlet boundary conditions  and $p(\bx) = \sin(2 \pi x)\cos(2\pi y)$ be the exact pressure. The corresponding forcing term is computed accordingly.  
    \item \emph{Polynomial test}: let $\Psi(\bx) = \psi(x)\psi(y)$ be the stream function defined as $\psi(z) = z^2(1-z)^2$. Then, the exact velocity field is defined as $\bu(\bx) = (\partial_y \Psi, -\partial_x \Psi)$, which satisfies homogeneous Dirichlet boundary conditions on the unit square. The exact pressure is taken to be  $p(\bx) = 0$. The corresponding forcing term is derived accordingly.
\end{itemize}
 Let us start the analysis from the sinusoidal test. Tables \ref{tab:ExSin:Re100:errors:v} and \ref{tab:ExSin:Re1000:errors:v} report the numerical errors for the velocity corresponding to $Re=100$ and $Re=1000$, respectively. The pressure errors for $Re=100$ and $Re=1000$ are presented in Table \ref{tab:ExSin:Re100:errors:p} and Table \ref{tab:ExSin:Re1000:errors:p}, respectively. 
 In all cases, the convergence rates delivered by the smooth solution $\bu$ are in accordance with Lemma \ref{theo:estimates:velocity:higher_regularity}. The convergence rates 
  for pressure also show good agreement with the expected asymptotic convergence rates.
 
 For larger Reynolds numbers, it can be observed  that the preasymptotic region grows with the Reynolds number. This behavior is due to the increasing dominance of the convective term, which leads to larger constants in the error estimates. 

Consequently, smaller meshsizes $h$ are needed to observe the asymptotic regime.
The same observations hold for the Polynomial test. We report in Table \ref{tab:velocity_Re100} and \ref{tab:velocity_Re1000}, the velocity errors for $Re=100$ and $Re=1000$, respectively. The pressure errors are presented in Table \ref{tab:pressure_Re100} for $Re=100$ and in Table \ref{tab:pressure_Re1000} for $Re=1000$. The convergence rates are in agreement with the theoretical findings.

    \input{velocity_Re100_sin}
    \input{pressure_Re100_sin}

    \input{velocity_Re1000_sin}
    \input{pressure_Re1000_sin}

    \input{velocity_Re100_stream}
    \input{pressure_Re100_stream}

    \input{velocity_Re1000_stream}
    \input{pressure_Re1000_stream}

%% file: velocity_Re100_sin.tex
\begin{table}[]
\centering
\caption{ \emph{ Sinusoidal test}.  $H^1$-error for the velocity and convergence rates using the discrete Smagorinsky model for $Re=100$ and for cartesian, hexagonal and triangular meshes.}
\label{tab:ExSin:Re100:errors:v}
\begin{tabular}{c c c c c c c c c}
\toprule
&\multicolumn{2}{c}{$k=2$} &&\multicolumn{2}{c}{$k=3$} &&\multicolumn{2}{c}{$k=4$}\\
\cline{2-3}\cline{5-6}\cline{8-9}
$h$ & Error & Rate && Error & Rate && Error & Rate\\ 
\toprule
\multicolumn{9}{c}{Cartesian} \\
\midrule
1.414e-01 & 2.682e-01 & - & & 1.167e-01 & - & & 1.064e-01 & - \\
7.071e-02 & 7.049e-02 & 1.93 & & 2.925e-02 & 2.00 & & 2.858e-02 & 1.90 \\
3.536e-02 & 1.784e-02 & 1.98 & & 7.329e-03 & 2.00 & & 7.286e-03 & 1.97 \\
1.768e-02 & 4.474e-03 & 2.00 & & 1.833e-03 & 2.00 & & 1.831e-03 & 1.99 \\
8.839e-03 & 1.119e-03 & 2.00 & & 4.585e-04 & 2.00 & & 4.583e-04 & 2.00 \\
\bottomrule
\multicolumn{9}{c}{Hexagonal} \\
\hline
2.635e-01 & 8.753e-01 & - & & 2.534e-01 & - & & 2.208e-01 & - \\
1.318e-01 & 2.132e-01 & 2.04 & & 8.144e-02 & 1.64 & & 7.497e-02 & 1.56 \\
6.588e-02 & 5.570e-02 & 1.94 & & 2.067e-02 & 1.98 & & 2.027e-02 & 1.89 \\
3.294e-02 & 1.405e-02 & 1.99 & & 5.177e-03 & 2.00 & & 5.153e-03 & 1.98 \\
1.647e-02 & 3.520e-03 & 2.00 & & 1.295e-03 & 2.00 & & 1.294e-03 & 1.99 \\
\bottomrule
\multicolumn{9}{c}{Triangular} \\
\hline
2.500e-01 & 4.995e-01 & - & & 2.981e-01 & - & & 2.401e-01 & - \\
1.250e-01 & 1.651e-01 & 1.60 & & 8.313e-02 & 1.84 & & 7.903e-02 & 1.60 \\
6.250e-02 & 4.938e-02 & 1.74 & & 2.020e-02 & 2.04 & & 1.994e-02 & 1.99 \\
3.125e-02 & 1.312e-02 & 1.91 & & 5.041e-03 & 2.00 & & 5.011e-03 & 1.99 \\
\bottomrule
\end{tabular}
\end{table}

%% file: pressure_Re100_sin.tex
\begin{table}[]
\centering
\caption{ \emph{ Sinusoidal test}.  $L^2$-errors for the pressure and convergence rates using the discrete Smagorinsky model for $Re=100$ and for cartesian, hexagonal and triangular meshes.}
\vspace{2mm}
\label{tab:ExSin:Re100:errors:p}
\begin{tabular}{c c c c c c c c c}
\toprule
&\multicolumn{2}{c}{$k=2$} &&\multicolumn{2}{c}{$k=3$} &&\multicolumn{2}{c}{$k=4$}\\
\cline{2-3}\cline{5-6}\cline{8-9}
$h$ & Error & Rate && Error & Rate && Error & Rate\\ \toprule
\multicolumn{9}{c}{Cartesian} \\
\midrule
1.414e-01 & 1.925e-02 & - & & 2.589e-03 & - & & 1.717e-03 & - \\
7.071e-02 & 4.871e-03 & 1.98 & & 5.291e-04 & 2.29 & & 4.689e-04 & 1.87 \\
3.536e-02 & 1.221e-03 & 2.00 & & 1.240e-04 & 2.09 & & 1.201e-04 & 1.96 \\
1.768e-02 & 3.055e-04 & 2.00 & & 3.047e-05 & 2.03 & & 3.022e-05 & 1.99 \\
8.839e-03 & 7.640e-05 & 2.00 & & 7.583e-06 & 2.01 & & 7.568e-06 & 2.00 \\
\bottomrule
\multicolumn{9}{c}{Hexagonal} \\
\hline
2.635e-01 & 5.074e-02 & - & & 8.988e-03 & - & & 3.408e-03 & - \\
1.318e-01 & 1.309e-02 & 1.95 & & 1.611e-03 & 2.48 & & 1.184e-03 & 1.53 \\
6.588e-02 & 3.303e-03 & 1.99 & & 3.486e-04 & 2.21 & & 3.202e-04 & 1.89 \\
3.294e-02 & 8.275e-04 & 2.00 & & 8.332e-05 & 2.06 & & 8.146e-05 & 1.97 \\
1.647e-02 & 2.070e-04 & 2.00 & & 2.058e-05 & 2.02 & & 2.045e-05 & 1.99 \\
\bottomrule
\multicolumn{9}{c}{Triangular} \\
\hline
2.500e-01 & 4.133e-02 & - & & 6.447e-03 & - & & 3.728e-03 & - \\
1.250e-01 & 9.958e-03 & 2.05 & & 1.485e-03 & 2.12 & & 1.281e-03 & 1.54 \\
6.250e-02 & 2.556e-03 & 1.96 & & 3.371e-04 & 2.14 & & 3.245e-04 & 1.98 \\
3.125e-02 & 6.457e-04 & 1.98 & & 8.273e-05 & 2.03 & & 8.193e-05 & 1.99 \\
\bottomrule
\end{tabular}
\end{table}

%% file: velocity_Re1000_sin.tex
\begin{table}[]
\centering
\caption{ \emph{ Sinusoidal test}.  $H^1$-error for the velocity and convergence rates using the discrete Smagorinsky model for $Re=1000$ and for cartesian, hexagonal and triangular meshes.}
\vspace{2mm}
\label{tab:ExSin:Re1000:errors:v}
\begin{tabular}{c c c c c c c c c}
\toprule
&\multicolumn{2}{c}{$k=2$} &&\multicolumn{2}{c}{$k=3$} &&\multicolumn{2}{c}{$k=4$}\\
\cline{2-3}\cline{5-6}\cline{8-9}
$h$ & Error & Rate && Error & Rate && Error & Rate\\ \toprule
\multicolumn{9}{c}{Cartesian} \\
\midrule
1.414e-01 & 9.371e-01 & - & & 8.622e-01 & - & & 8.556e-01 & - \\
7.071e-02 & 3.444e-01 & 1.44 & & 3.320e-01 & 1.38 & & 3.318e-01 & 1.37 \\
3.536e-02 & 8.909e-02 & 1.95 & & 8.734e-02 & 1.93 & & 8.733e-02 & 1.93 \\
1.768e-02 & 2.259e-02 & 1.98 & & 2.220e-02 & 1.98 & & 2.220e-02 & 1.98 \\
8.839e-03 & 5.670e-03 & 1.99 & & 5.577e-03 & 1.99 & & 5.577e-03 & 1.99 \\
\bottomrule
\multicolumn{9}{c}{Hexagonal} \\
\hline
2.635e-01 & 1.814e+00 & - & & 1.634e+00 & - & & 1.622e+00 & - \\
1.318e-01 & 7.639e-01 & 1.25 & & 9.257e-01 & 0.82 & & 8.852e-01 & 0.87 \\
6.588e-02 & 2.440e-01 & 1.65 & & 2.321e-01 & 2.00 & & 2.321e-01 & 1.93 \\
3.294e-02 & 6.180e-02 & 1.98 & & 6.017e-02 & 1.95 & & 6.017e-02 & 1.95 \\
1.647e-02 & 1.562e-02 & 1.98 & & 1.529e-02 & 1.98 & & 1.530e-02 & 1.98 \\
\bottomrule
\multicolumn{9}{c}{Triangular} \\
\hline
2.500e-01 & 1.096e+00 & - & & 1.341e+00 & - & & 1.619e+00 & - \\
1.250e-01 & 7.823e-01 & 0.49 & & 7.341e-01 & 0.87 & & 7.347e-01 & 1.14 \\
6.250e-02 & 2.271e-01 & 1.78 & & 2.220e-01 & 1.73 & & 2.221e-01 & 1.73 \\
3.125e-02 & 6.100e-02 & 1.90 & & 5.965e-02 & 1.90 & & 5.965e-02 & 1.90 \\
\bottomrule
\end{tabular}
\end{table}

%% file: pressure_Re1000_sin.tex
\begin{table}[H]
\centering
\caption{ \emph{ Sinusoidal test}.  $L^2$-errors for the pressure and convergence rates using the discrete Smagorinsky model for $Re=1000$ and for cartesian, hexagonal and triangular meshes.}
\vspace{2mm}
\label{tab:ExSin:Re1000:errors:p}
\begin{tabular}{c c c c c c c c c}
\toprule
&\multicolumn{2}{c}{$k=2$} &&\multicolumn{2}{c}{$k=3$} &&\multicolumn{2}{c}{$k=4$}\\
\cline{2-3}\cline{5-6}\cline{8-9}
$h$ & Error & Rate && Error & Rate && Error & Rate\\ \toprule
\multicolumn{9}{c}{Cartesian} \\
\midrule
1.414e-01 & 2.188e-02 & - & & 1.052e-02 & - & & 1.032e-02 & - \\
7.071e-02 & 6.569e-03 & 1.74 & & 4.396e-03 & 1.26 & & 4.388e-03 & 1.23 \\
3.536e-02 & 1.747e-03 & 1.91 & & 1.254e-03 & 1.81 & & 1.254e-03 & 1.81 \\
1.768e-02 & 4.463e-04 & 1.97 & & 3.266e-04 & 1.94 & & 3.266e-04 & 1.94 \\
8.839e-03 & 1.122e-04 & 1.99 & & 8.256e-05 & 1.98 & & 8.256e-05 & 1.98 \\
\bottomrule
\multicolumn{9}{c}{Hexagonal} \\
\hline
2.635e-01 & 5.237e-02 & - & & 1.762e-02 & - & & 1.575e-02 & - \\
1.318e-01 & 1.563e-02 & 1.74 & & 1.017e-02 & 0.79 & & 9.841e-03 & 0.68 \\
6.588e-02 & 4.543e-03 & 1.78 & & 3.115e-03 & 1.71 & & 3.113e-03 & 1.66 \\
3.294e-02 & 1.187e-03 & 1.94 & & 8.562e-04 & 1.86 & & 8.561e-04 & 1.86 \\
1.647e-02 & 3.016e-04 & 1.98 & & 2.206e-04 & 1.96 & & 2.206e-04 & 1.96 \\
\bottomrule
\multicolumn{9}{c}{Triangular} \\
\hline
2.500e-01 & 4.277e-02 & - & & 1.504e-02 & - & & 1.580e-02 & - \\
1.250e-01 & 1.346e-02 & 1.67 & & 8.996e-03 & 0.74 & & 8.974e-03 & 0.82 \\
6.250e-02 & 3.952e-03 & 1.77 & & 3.082e-03 & 1.55 & & 3.081e-03 & 1.54 \\
3.125e-02 & 1.062e-03 & 1.90 & & 8.637e-04 & 1.84 & & 8.636e-04 & 1.83 \\
\bottomrule
\end{tabular}
\end{table}

%% file: velocity_Re100_stream.tex
\begin{table}[]
\centering
\caption{ \emph{ Polynomial test}.  $H^1$-error for the velocity and convergence rates using the discrete Smagorinsky model for $Re=100$ and for cartesian, hexagonal and triangular meshes.}
\vspace{2mm}
\label{tab:velocity_Re100}
\begin{tabular}{c c c c c c c c c}
\toprule
&\multicolumn{2}{c}{$k=2$} &&\multicolumn{2}{c}{$k=3$} &&\multicolumn{2}{c}{$k=4$}\\
\cline{2-3}\cline{5-6}\cline{8-9}
$h$ & Error & Rate && Error & Rate && Error & Rate\\ \toprule 
\multicolumn{9}{c}{Cartesian} \\
\midrule
1.414e-01 & 1.724e-03 & - & & 1.882e-04 & - & & 8.102e-05 & - \\
7.071e-02 & 4.381e-04 & 1.98 & & 2.953e-05 & 2.67 & & 2.006e-05 & 2.01 \\
3.536e-02 & 1.100e-04 & 1.99 & & 5.705e-06 & 2.37 & & 5.014e-06 & 2.00 \\
1.768e-02 & 2.752e-05 & 2.00 & & 1.299e-06 & 2.13 & & 1.254e-06 & 2.00 \\
8.839e-03 & 6.882e-06 & 2.00 & & 3.163e-07 & 2.04 & & 3.134e-07 & 2.00 \\
\bottomrule
\multicolumn{9}{c}{Hexagonal} \\
\hline
2.635e-01 & 5.189e-03 & - & & 9.034e-04 & - & & 3.236e-04 & - \\
1.318e-01 & 1.383e-03 & 1.91 & & 1.461e-04 & 2.63 & & 8.371e-05 & 1.95 \\
6.588e-02 & 3.522e-04 & 1.97 & & 2.657e-05 & 2.46 & & 2.173e-05 & 1.95 \\
3.294e-02 & 8.842e-05 & 1.99 & & 5.835e-06 & 2.19 & & 5.521e-06 & 1.98 \\
1.647e-02 & 2.212e-05 & 2.00 & & 1.409e-06 & 2.05 & & 1.390e-06 & 1.99 \\
\bottomrule
\multicolumn{9}{c}{Triangular} \\
\hline
2.500e-01 & 3.826e-03 & - & & 6.747e-04 & - & & 2.484e-04 & - \\
1.250e-01 & 1.211e-03 & 1.66 & & 9.569e-05 & 2.82 & & 5.820e-05 & 2.09 \\
6.250e-02 & 3.151e-04 & 1.94 & & 1.791e-05 & 2.42 & & 1.421e-05 & 2.03 \\
3.125e-02 & 7.945e-05 & 1.99 & & 4.066e-06 & 2.14 & & 3.511e-06 & 2.02 \\
\bottomrule
\end{tabular}
\end{table}

%% file: pressure_Re100_stream.tex
\begin{table}[]
\centering
\caption{ \emph{ Polynomial test}.  $L^2$-errors for the pressure and convergence rates using the discrete Smagorinsky model for $Re=100$ and for cartesian, hexagonal and triangular meshes.}
\vspace{2mm}
\label{tab:pressure_Re100}
\begin{tabular}{c c c c c c c c c}
\toprule
&\multicolumn{2}{c}{$k=2$} &&\multicolumn{2}{c}{$k=3$} &&\multicolumn{2}{c}{$k=4$}\\
\cline{2-3}\cline{5-6}\cline{8-9}
$h$ & Error & Rate && Error & Rate && Error & Rate\\ \toprule
\multicolumn{9}{c}{Cartesian} \\
\midrule
1.414e-01 & 7.749e-07 & - & & 1.896e-07 & - & & 1.575e-07 & - \\
7.071e-02 & 1.072e-07 & 2.85 & & 4.015e-08 & 2.24 & & 3.872e-08 & 2.02 \\
3.536e-02 & 1.589e-08 & 2.75 & & 9.716e-09 & 2.05 & & 9.664e-09 & 2.00 \\
1.768e-02 & 2.886e-09 & 2.46 & & 2.417e-09 & 2.01 & & 2.416e-09 & 2.00 \\
8.839e-03 & 6.353e-10 & 2.18 & & 6.039e-10 & 2.00 & & 6.039e-10 & 2.00 \\
\bottomrule
\multicolumn{9}{c}{Hexagonal} \\
\hline
2.635e-01 & 4.519e-05 & - & & 2.424e-06 & - & & 1.952e-06 & - \\
1.318e-01 & 7.371e-06 & 2.62 & & 6.098e-07 & 1.99 & & 4.995e-07 & 1.97 \\
6.588e-02 & 1.457e-06 & 2.34 & & 1.426e-07 & 2.10 & & 1.207e-07 & 2.05 \\
3.294e-02 & 3.255e-07 & 2.16 & & 3.404e-08 & 2.07 & & 2.907e-08 & 2.05 \\
1.647e-02 & 7.705e-08 & 2.08 & & 8.286e-09 & 2.04 & & 7.084e-09 & 2.04 \\
\bottomrule
\multicolumn{9}{c}{Triangular} \\
\hline
2.500e-01 & 5.124e-05 & - & & 5.667e-06 & - & & 7.755e-07 & - \\
1.250e-01 & 1.788e-05 & 1.52 & & 5.997e-07 & 3.24 & & 1.563e-07 & 2.31 \\
6.250e-02 & 4.817e-06 & 1.89 & & 7.626e-08 & 2.98 & & 3.706e-08 & 2.08 \\
3.125e-02 & 1.224e-06 & 1.98 & & 1.303e-08 & 2.55 & & 9.132e-09 & 2.02 \\
\bottomrule
\end{tabular}
\end{table}

%% file: velocity_Re1000_stream.tex
\begin{table}[]
\centering
\caption{ \emph{ Polynomial test}.  $H^1$-error for the velocity and convergence rates using the discrete Smagorinsky model for $Re=1000$ and for cartesian, hexagonal and triangular meshes.}
\vspace{2mm}
\label{tab:velocity_Re1000}
\begin{tabular}{c c c c c c c c c}
\toprule
&\multicolumn{2}{c}{$k=2$}  &&\multicolumn{2}{c}{$k=3$}  &&\multicolumn{2}{c}{$k=4$} \\
\cline{2-3}\cline{5-6}\cline{8-9}
$h$ & Error & Rate && Error & Rate && Error & Rate\\ \toprule
\multicolumn{9}{c}{Cartesian} \\
\midrule
1.414e-01 & 1.885e-03 & - & & 7.977e-04 & - & & 7.796e-04 & - \\
7.071e-02 & 4.806e-04 & 1.97 & & 2.004e-04 & 1.99 & & 1.992e-04 & 1.97 \\
3.536e-02 & 1.207e-04 & 1.99 & & 5.015e-05 & 2.00 & & 5.008e-05 & 1.99 \\
1.768e-02 & 3.022e-05 & 2.00 & & 1.254e-05 & 2.00 & & 1.254e-05 & 2.00 \\
8.839e-03 & 7.556e-06 & 2.00 & & 3.136e-06 & 2.00 & & 3.135e-06 & 2.00 \\
\bottomrule
\multicolumn{9}{c}{Hexagonal} \\
\hline
2.635e-01 & 5.512e-03 & - & & 2.063e-03 & - & & 1.876e-03 & - \\
1.318e-01 & 1.488e-03 & 1.89 & & 5.420e-04 & 1.93 & & 5.281e-04 & 1.83 \\
6.588e-02 & 3.804e-04 & 1.97 & & 1.390e-04 & 1.96 & & 1.382e-04 & 1.93 \\
3.294e-02 & 9.561e-05 & 1.99 & & 3.503e-05 & 1.99 & & 3.499e-05 & 1.98 \\
1.647e-02 & 2.393e-05 & 2.00 & & 8.778e-06 & 2.00 & & 8.778e-06 & 2.00 \\
\bottomrule
\multicolumn{9}{c}{Triangular} \\
\hline
2.500e-01 & 4.401e-03 & - & & 2.262e-03 & - & & 2.174e-03 & - \\
1.250e-01 & 1.338e-03 & 1.72 & & 5.659e-04 & 2.00 & & 5.608e-04 & 1.96 \\
6.250e-02 & 3.445e-04 & 1.96 & & 1.397e-04 & 2.02 & & 1.393e-04 & 2.01 \\
3.125e-02 & 8.657e-05 & 1.99 & & 3.460e-05 & 2.01 & & 3.454e-05 & 2.01 \\
\bottomrule
\end{tabular}
\end{table}

%% file: pressure_Re1000_stream.tex
\begin{table}[]
\centering
\caption{ \emph{ Polynomial test}.  $L^2$-errors for the pressure and convergence rates using the discrete Smagorinsky model for $Re=1000$ and for cartesian, hexagonal and triangular meshes.}
\vspace{2mm}
\label{tab:pressure_Re1000}
\begin{tabular}{c c c c c c c c c}
\toprule
&\multicolumn{2}{c}{$k=2$}  &&\multicolumn{2}{c}{$k=3$}  &&\multicolumn{2}{c}{$k=4$} \\
\cline{2-3}\cline{5-6}\cline{8-9}
$h$ & Error & Rate && Error & Rate && Error & Rate\\ \toprule
\multicolumn{9}{c}{Cartesian} \\
\midrule
1.414e-01 & 1.104e-06 & - & & 1.092e-06 & - & & 1.092e-06 & - \\
7.071e-02 & 2.806e-07 & 1.98 & & 2.799e-07 & 1.96 & & 2.799e-07 & 1.96 \\
3.536e-02 & 7.046e-08 & 1.99 & & 7.042e-08 & 1.99 & & 7.042e-08 & 1.99 \\
1.768e-02 & 1.764e-08 & 2.00 & & 1.763e-08 & 2.00 & & 1.763e-08 & 2.00 \\
8.839e-03 & 4.410e-09 & 2.00 & & 4.410e-09 & 2.00 & & 4.410e-09 & 2.00 \\
\bottomrule
\multicolumn{9}{c}{Hexagonal} \\
\hline
2.635e-01 & 5.282e-06 & - & & 2.670e-06 & - & & 2.676e-06 & - \\
1.318e-01 & 1.050e-06 & 2.33 & & 7.378e-07 & 1.86 & & 7.379e-07 & 1.86 \\
6.588e-02 & 2.409e-07 & 2.12 & & 1.893e-07 & 1.96 & & 1.892e-07 & 1.96 \\
3.294e-02 & 5.824e-08 & 2.05 & & 4.764e-08 & 1.99 & & 4.761e-08 & 1.99 \\
1.647e-02 & 1.434e-08 & 2.02 & & 1.193e-08 & 2.00 & & 1.192e-08 & 2.00 \\
\bottomrule
\multicolumn{9}{c}{Triangular} \\
\hline
2.500e-01 & 5.856e-06 & - & & 3.042e-06 & - & & 2.986e-06 & - \\
1.250e-01 & 1.944e-06 & 1.59 & & 7.681e-07 & 1.99 & & 7.660e-07 & 1.96 \\
6.250e-02 & 5.187e-07 & 1.91 & & 1.920e-07 & 2.00 & & 1.919e-07 & 2.00 \\
3.125e-02 & 1.316e-07 & 1.98 & & 4.802e-08 & 2.00 & & 4.801e-08 & 2.00 \\
\bottomrule
\end{tabular}
\end{table}

%% file: conclusions.tex
\section{Conclusions}
In this work, we have extended the analysis of incompressible fluid-flows within the virtual element framework to include nonlinear eddy viscosity models. We restricted our attention to the classical Smagorinsky  model that represents the first step toward more comprehensive models in convection-dominated flows. Specifically, we have analyzed the Navier-Stokes-Smagorinsky equations combined with the divergence-free VEM discretization. Under the usual small data assumption, we have provided the classical convergence rates $h$ for \textit{a priori} error estimates. Sharper convergence is obtained with weaker regularity assumptions than those in FEM, which are validated through numerical tests that support the theory. This research opens the way to future investigations as for example
more advanced models taking into account anisotropy or energy transfer from small to large scales.

%% file: acknowledgements.tex
\section*{Acknowledgments}
\emergencystretch=3em 
This manuscript reflects only the authors’ views and opinions, and the Ministry cannot be considered responsible for them.
The authors  acknowledge the financial support by INdAM-GNCS
through the projects 2025 (CUP: E53C24001950001) and 2026 (CUP: E53C25002010001).